%% file: revised.tex
\documentclass{article}
\usepackage{graphicx}
\pdfoutput=1

\include{macro}

\title{Entropy of Soft Random Geometric Graphs in General Geometries}
\author[1,$\dagger$]{Oliver Baker}
\author[1]{Carl P. Dettmann}

\affil[1]{School of Mathematics, University of Bristol\footnote{Email: \{o.baker, carl.dettmann\}@bristol.ac.uk}}
\affil[$\dagger$]{Corresponding Author}
\date{\today}

\renewbibmacro{in:}{}
\DeclareFieldFormat{pages}{#1}
\DeclareFieldFormat[article, inproceedings]{volume}{\mkbibbold{#1}}

\begin{document}

\maketitle

\begin{abstract}
    We study the effect of the choice of embedding geometry on the entropy of random geometric graph ensembles with soft connection functions. First we show that when the connection range is small, to leading order the entropy is dependent only on the dimension of the geometry and not the shape, with shape effects entering at the next term, but for large connection ranges the boundaries of the domain matter. Next, we formulate the problem of estimating entropy as a problem of estimating the average degree of a graph with the binary entropy function as its connection function. We use this formulation to study the effect of boundaries on the entropy, and to estimate the entropy of soft random geometric graphs in complicated geometries where a closed form pair distance density is not available.
\end{abstract}

% \tableofcontents

\section{Introduction}
The study of random graphs has found applications in a wide range of fields, 
including sociology \cite{watts1998collective}, wireless communications \cite{haenggi2009stochastic} and biology \cite{van2024random}. As the availability of data increases, these networks become larger, and it becomes more important to understand how local properties reflect the macroscopic properties of the network. A large class of networks exhibit a spatial embedding \cite{barthelemy2011spatial}. One model of a spatial network is the Random Geometric Graph (RGG), first introduced as the `Random Planar Network' by Gilbert in 1961 \cite{gilbert1961random}. In this model, randomly distributed nodes are connected if they lie less than a fixed distance apart. Waxman extended this model to allow for probabilistic connections that decay exponentially in distance, with a view to modelling wireless communication networks \cite{waxman1988routing}. \\

Waxman's model can be modified to allow for more general connection functions. We call this model the \textit{Soft} Random Geometric Graph (SRGG), which is a popular choice when modelling wireless networks \cite{haenggi2009stochastic}. Recent developments in the SRGG literature include studies of connectivity \cite{penrose2016connectivity, dettmann2016random} and entropy \cite{coon2018entropy,vippathalla2024lossy,baker2025entropy}. In this work, we focus on the latter. The entropy of the SRGG is fundamental to the compression of random networks \cite{vippathalla2024lossy}, and how efficiently nodes can communicate topological information around a communication network \cite{an2002entropy, coon2016topological}. There has been recent interest concerning the \textit{structural} entropy of random graphs, \cite{compression_choi_2012,asymmetry_uczak_2019,structural_badiu_2021}, which is equivalent to the entropy of an \textit{unlabelled} random graph \cite{paton2022entropy}. In the dense regime in which this paper operates, the entropy of the unlabelled graph and the labelled graph (the object of study here), are asymptotically equivalent, since the labels contribute a lower order of entropy than the graph structure. In statistical physics, the canonical ensemble of spatial networks is the one that maximises the entropy with respect to a set of constraints. This results in an SRGG ensemble whose connection function takes a Fermi-Dirac form \cite{bianconi2021information, krioukov2016clustering, coon2018entropy}. \\

In this work, we quantify the effect of the embedding geometry on the entropy of the SRGG. In Section \ref{sec:scaling}, we prove the following
\begin{enumerate}
    \item A universality result for SRGG ensembles in the small connection range limit ($r_0 \to 0$), where the entropy behaves asymptotically as $C_{p,d}r_0^d$ where $C_{p,d}$ is a constant dependent only on the dimension and the connection function and $d$ is the dimension of the space. Shape effects enter at the next order,
    \item A scaling law for the entropy in the large connection range limit, where the embedding geometry \textit{does} contribute
    \item These scaling laws imply a qualitative gap in compressibility between SRGG and Erd\H os–R\'enyi ensembles in the short-range regime.
\end{enumerate}

In Section \ref{sec:entropy_graph}, we construct an auxiliary graph, called the `entropy graph' so that the average degree of this graph equals the conditional entropy of the SRGG. We use this formalism and methods from the connectivity literature (\cite{dettmann2016random}) to quantify the effects of each boundary component on the conditional entropy. As two examples, we quantify the effect of a corner using a wedge geometry, and demonstrate log-periodic oscillatory behaviour of SRGG entropy in the Cantor set. \\

Section \ref{sec:conclusion} concludes, Appendix \ref{sec:proofs} contains proofs of the results in Section \ref{sec:scaling}, and Appendix \ref{sec:wedge_and_cantor} contains the proofs and derivations from Section \ref{sec:entropy_graph}.

\section{Preliminaries}

\subsection{Asymptotic notation}

We begin with a brief overview of the asymptotic notation used in this paper. Given two real-valued functions $f$ and $g$, we say $f(x) = \bigO{g(x)}$ as $x\tti$ if there exist constants $c\in\mathbb{R}_{>0}$, and $x_0$ in the domain of $f$ and $g$ such that $f(x) \leq cg(x)$ for all $x > x_0$. Likewise, $f(x) = \bigO{g(x)}$ as $x \ttz$ if $f(x) \leq cg(x)$ for all $x < x_0$. If $f(x) = \bigO{g(x)}$ and $g(x) = \bigO{f(x)}$ then we write $f(x) = \Theta(g(x))$ (or equivalently, $g(x) = \Theta(f(x))$. If $\lim_{x\tti} f(x)/g(x) = 0$, then we say $f(x) = o(g(x))$ as $x \tti$, and if $\lim_{x\ttz} f(x)/g(x) = 0$, then we say $f(x) = o(g(x))$ as $x\ttz$. Finally, if $\lim_{x\tti} f(x)/g(x) = 1$, then we write $f(x) \sim g(x)$ as $x\tti$ (and similarly for $x\ttz$).

\subsection{Soft Random Geometric Graphs}

Let $(\Omega, \mathcal{F}, \mu)$ be a probability space equipped with a metric induced by the Euclidean norm $\|\cdot\|$, and $\Omega \subset \mathbb{R}^d$. Except for Section \ref{cantor}, the measure $\mu$ will represent the Lebesgue measure. Define $D = \sup_{x,y \in \Omega}\|x-y\|$, the \textit{diameter} of the domain $\Omega$. A \textit{Soft Random Geometric Graph} (SRGG) is a graph $G = (V,E)$ where $V = \{1,...,n\}$, and $E$ is constructed as follows. Distribute points $\{Z_i\}_{i\in V}$ in $\Omega$ according to probability measure $\mu$. Let $r_{ij} := \|Z_i-Z_j\|$ and $p : \mathbb{R}_{\geq 0}  \rightarrow [0,1]$, then we will include each edge $(i,j)$ in the edge set $E$ independently with probability $p(r_{ij}/r_0)$ where $r_0 \in \mathbb{R}$ is known as the \textit{typical connection range}. If we take $p(r/r_0) = \mathbb{I}(r/r_0<1)$ then we recover the classical \textit{hard} random geometric graph, and if we take $p(r/r_0) = q$ for a constant $q$, the SRGG is equivalent to the Erd\H os-R\'enyi (ER) random graph $G(n,q)$. \newline

There are several standard choices for the connection function $p$. In the wireless communications literature, a standard choice is the connection function
\begin{equation}
    \label{rayleigh}
    p(r/r_0) = \rayleigh{r}{\eta}
\end{equation}
which is known as the \textit{Rayleigh fading} model, and is derived through the signal to noise ratio of a wireless channel (see e.g. \cite{haenggi2009stochastic}), which is governed by the \textit{path-loss exponent} $\eta$ which tends to be between 2 and 6 in practical applications. This model is also of theoretical interest because $\eta = 0$ gives the ER graph, and $\eta \tti$ recovers the hard RGG. Other popular choices of connection function in other fields include the Fermi-Dirac connection function $p(r/r_0) = (1+\exp(c_0 + r/r_0))^{-1}$ which is derived from maximum entropy statistics of the SRGG model \cite{bianconi2009entropy}, and \textit{non-homogeneous} connection functions like $\tilde{p}(x/r_0, y/r_0) = W_xW_y\min\left(1, \left(\frac{\|x-y\|}{r_0}\right)^{-\nu}\right)$ where $W_x,W_y$ are \textit{weights} associated with $x$ and $y$. Connection functions like $\tilde{p}$ can give us a \textit{power-law} degree distribution (see e.g. \cite{gracar2022chemical}). Many other connection functions are explored in \cite{dettmann2016random}. \newline

To fix some notation, we denote by $f(r)$ the probability density function of the distance between two random points distributed according to $\mu$ in $\Omega$, $\mathcal{R} \in \mathbb{R}^{\binom{n}{2}}$ to be the random vector representing the distances between each pair of $n$ random points in $\Omega$. We note that the function $f$ is not well defined for the Cantor set, see Section \ref{cantor}. $\graph(n,r_0,\eta)$ is the ensemble of SRGGs on $n$ nodes with a Rayleigh fading connection function with parameters $r_0$ and $\eta$. We will abbreviate this to $\graph$ when the parameters are either clear from context or unimportant in the context. Finally, denote by $X_{ij}$ the Bernoulli random variable which represents the existence of edge $(i,j)$ in the SRGG.

In general we will consider connection functions of the form $p= p(r/r_0)$, i.e. those that take the ratio of the distance and connection range as the argument. The exception is in Section \ref{sec:entropy_mass}, where we will make the dependence on $r_0$ implicit and write for example $p(r) = \exp(-(r/r_0)^\eta)$. This is because in this section, we will treat connection functions with different connection range parameters as different connection functions, rather than treating $r_0$ as a parameter that we can take limits with respect to.

\subsection{Entropy of Soft Random Geometric Graphs}

If $\graph$ is an ensemble of random graphs on $n$ nodes (not necessarily SRGGs), we define the \textit{entropy} of the ensemble as
\begin{equation}
    H(\graph) := -\sum_{G \in \graph} \prob{G}\log \prob{G}
\end{equation}
where the logarithm is to base $e$. In \cite{coon2018entropy} some simple bounds were derived for the entropy of the SRGG ensemble. First, we must note that the entropy of the ensemble $G$ is equivalent to the joint entropy of the edge variables, $H(\{X_{ij}\}_{i<j})$ where $X_{ij} = 1$ if $(i,j)\in E$ and is 0 otherwise. Then, using Shannon's independence bound we obtain
\begin{equation}
    H(\graph) \leq \sum_{i<j} h_2(\prob{X_{ij} = 1}) = \binom{n}{2}h_2(\bar{p})
\end{equation}
where $h_2(p) = -p\log p - (1-p)\log(1-p)$ is the binary entropy function, and $\bar{p} = \int_{\Omega}\int_{\Omega} p(\|x-y\|/r_0) dxdy$ is the \textit{average connection probability}. We may also write $\bar{p} = \int_0^D f(r)p(r)dr$, where $f$ is the density function of the distance between two random points in $\Omega$. As a lower bound, we condition on the locations of nodes, or equivalently on the pair distances between nodes. We refer to \cite{coon2018entropy} for a more in-depth explanation.
\begin{equation}
    \label{lower_bound}
    H(\graph) \geq H(\graph(r_0)|\mathcal{R})= \binom{n}{2}\int_{\Omega}\int_{\Omega} h_2(p(\|x-y\|/r_0)) d\mu(x)d\mu(y)
\end{equation}
It was shown in \cite[Theorem D.5]{janson2010graphons} that, if the function $W : \Omega^2 \to [0,1]$, $W(x,y) = p(\|x-y\|/r_0)$ is measurable, then (\ref{lower_bound}) is also the limiting entropy-per-edge of the SRGG ensemble as $n\tti$. Specifically,
\begin{equation}
    \lim_{n\rightarrow\infty} \frac{H(\graph)}{\binom{n}{2}} = \int_{\Omega}\int_{\Omega} h_2(p(\|x-y\|/r_0)) d\mu(x)d\mu(y)
\end{equation}
% This result is also valid for hard RGGs, except that the right hand side is zero, because the hard RGG is deterministic given a fixed node embedding. This is consistent with what is known about the entropy of the hard RGG, for example in \cite{hatami2013entropy} it is shown that any graph connection function that is $\{0,1\}$-valued must have entropy $\bigO{n\log n}$ as $n\tti$. It was also shown non-rigorously in \cite{krioukov2016clustering} that for an SRGG with linearly growing support, for example $\Omega_n = [-n/2, n/2]$ with toroidal boundary conditions (creating a sparse graph) then we have
% \begin{equation}
%     \lim_{n\tti} \frac{H(\graph)}{n} = \frac{1}{2}\int\int_{\mathbb{R}^2} h_2(p(\|x-y\|/r_0))dxdy
% \end{equation}
These results motivate that $H(\graph(r_0)|\mathcal{R})$ is an important object to study in the field of SRGG entropy. We note that in the dense regime, the entropy of the labelled graph (which we consider here) is, to leading order, the same as the entropy of the unlabelled graph. To see this, note that the result above gives $H(\graph) = \Theta(n^2)$. Also, the labelled entropy satisfies $H(\graph) \sim H(\mathcal{S}) + n\log n$, where $H(\mathcal{S})$ is the structural entropy of the SRGG, and, from Stirling's approximation, the $n\log n$ term is the entropy contribution from $n!$ rearrangements of the node labels \cite{structural_badiu_2021, paton2022entropy}. Therefore, upon normalisation by $n^2$, we see that the labelled and unlabelled entropy tend to the same limit as $n\tti$. To ease notation, we will introduce the conditional entropy-per-edge of the SRGG, which we denote as $\overline{H}(\graph(r_0)|\mathcal{R}) := \binom{n}{2}^{-1}H(\graph(r_0)|\mathcal{R})$. We will investigate first how this quantity scales in the limiting cases of very small and very large typical connection ranges.

\section{Entropy Scaling in the Connection Range}
\label{sec:scaling}
\subsection{Small Connection Range}
    In this section, we will analyse the initial slope of the conditional entropy of SRGGs. Whilst we restrict the connection function here to the case of Rayleigh fading, we maintain significant generality in the domain $\Omega$. We will assume only the following, which will give us an asymptotic expression for the distance distribution for small $r$.
    \begin{assumption}
        \label{assumption_omega}
        The domain $\Omega \subset \mathbb{R}^d$
        \begin{itemize}
            \item[A1.] has a non-empty interior
            \item[A2.] has a boundary that is piecewise smooth, has a finite number of corners, and is of finite length.
        \end{itemize}
    \end{assumption}
    
    We also stress that, with $D := \sup\{\|x-y\| : x,y \in \Omega\}$ as the diameter of $\Omega$ and $f(r)$ as the density of pairwise distances in $\Omega$, the integral in (\ref{lower_bound}) can be written as 
    \begin{equation}
        \label{eq:entropy_per_edge}
        \overline{H}(\graph(r_0,\eta)|\mathcal{R}) = \int_0^D f(r)h_2(p(r/r_0)) dr
    \end{equation}
    which, as in \cite{coon2018entropy}, is (once multiplied by $\binom{n}{2}$) the entropy of an SRGG conditioned on the node positions. Note that we have dropped the dependence on $n$, since this quantity is independent of the number of nodes. Therefore, as well as graphs in the limit of high density, the result below will hold for the entropy of an SRGG where we know the node locations. We will now prove that for small connection ranges $r_0$, the conditional entropy of SRGGs with Rayleigh connection functions scales as $\Theta(r_0^d)$, where $d$ is the dimension of the domain. More precisely:

    \begin{theorem}
        \label{small_r0_asymptotic}
        Let $\Omega$ satisfy the conditions of Assumption \ref{assumption_omega}, and $0 < \eta < \infty$. Then
        \begin{equation}
            \overline{H}(\graph(r_0,\eta)|\mathcal{R}) \sim \frac{s_{d-1}r_0^d}{\eta}\Gamma\left(\frac{d}{\eta}\right)\left(\frac{d}{\eta}+\zeta\left(\frac{d}{\eta}+1\right)-\sum_{k=2}^{\infty}\frac{k^{-\frac{d}{\eta}}}{k-1}\right)
        \end{equation}
        as $r_0 \ttz$, where $s_{d-1}$ is the surface area of the unit $d$-ball, $\Gamma$ is the gamma function and $\zeta$ is the Riemann-Zeta function.
    \end{theorem}
    \begin{proof}
        See Appendix \ref{small_r0_proof}.
    \end{proof}
    The result of this theorem is predicated on the fact that for a suitably regular domain, the small-scale distance density ($f(r)$ for small $r$) behaves like $s_{d-1}r^{d-1}$. For some geometries, we are able to also find the next highest term in the expansion. For example, when $\Omega = [0,1]$, $f(r) = 2-2r$. This means that we can obtain a better approximation to the conditional entropy. Following exactly the same procedure as in the proof of Theorem \ref{small_r0_asymptotic}, we find that for $r_0$ small, 
    \begin{align}
            \label{higher_order}
            \overline{H}(\graph(r_0,\eta)|\mathcal{R}) \approx \frac{s_{d-1}r_0^d}{\eta}\Gamma\left(\frac{d}{\eta}\right)\left(\frac{d}{\eta} + \zeta\left(\frac{d}{\eta}+1\right)-\sum_{k=2}^{\infty}\frac{k^{-\frac{d}{\eta}}}{k-1}\right) \nonumber \\
            + \frac{a_dr_0^{d+1}}{\eta}\Gamma\left(\frac{d+1}{\eta}\right)\left(\frac{d+1}{\eta}+\zeta\left(\frac{d+1}{\eta}+1\right)-\sum_{k=2}^{\infty}\frac{k^{-\frac{d+1}{\eta}}}{k-1}\right)
    \end{align}
    where $a_d$ is the coefficient of $r^d$ in $f(r)$. The asymptotic in the conclusion of Theorem \ref{small_r0_asymptotic} is independent of the choice of $\Omega$, except for the dimension of the space. This kind of result is not unique to the Rayleigh fading connection function. The main idea in the proof is to approximate the pair distance density $f(r)$ with the leading order monomial $f(r) = s_{d-1}r^{d-1}$. The coefficient in Theorem \ref{small_r0_asymptotic} comes from the expression
    \begin{equation}
        \overline{H}(\graph(r_0)|\mathcal{R}) \sim \int_0^{\sqrt{r_0}} s_{d-1}r^{d-1} h_2(p(r/r_0)) dr
    \end{equation}
    After a change of variables $t = r/r_0$ we find that we will see the dimension-dependent growth if and only if
    \begin{equation}
        \label{condition}
        \frac{1}{s_{d-1}r_0^d}\overline{H}(\graph(r_0)|\mathcal{R}) \sim \int_0^{\frac{1}{\sqrt{r_0}}} t^{d-1}h_2(p(t)) dt \in (0,\infty)
    \end{equation}
    As an example, the Rayleigh fading connection functions satisfy this property, as does the Fermi-Dirac connection function. So by the same reasoning as in Theorem 1, there exists a constant $C_{p,d}$ so that the Fermi-Dirac connection function behaves as $C_{p.d}r_0^d$ as $r_0\to0$. The power-law connection function behaves differently. For the power-law connection function
    \begin{equation}
        p_{PL}(r/r_0) = \min(1,(r/r_0)^{-\nu})
    \end{equation}
    the integral in (\ref{condition}) is
    \begin{equation}
        \int_1^{\frac{1}{\sqrt{r_0}}} t^{d-1}h_2(p_{PL}(t))dt = \int_1^{\frac{1}{\sqrt{r_0}}} t^{d-1}(-t\log t^{-\nu} - (1-t^{-\nu})\log(1-t^{-\nu}))dt
    \end{equation}
    \begin{equation}
        \geq \int_1^{\frac{1}{\sqrt{r_0}}} \nu t^{d-\nu-1}\log tdt
    \end{equation}
    which diverges whenever $d \geq \nu$. We also have
    \begin{gather}
        \label{eq:pl_bound_final}
        \int_1^{1/\sqrt{r_0}} t^{d-1}h_2(p_{PL}(t))dt = \int_1^{1/\sqrt{r_0}} t^{d-1}h_2(t^{-\nu})dt \nonumber \\ \leq \int_1^{1/\sqrt{r_0}} t^{d-1}(t^{-\nu} - t^{-\nu}\log t^{-\nu})= \int_1^{1/\sqrt{r_0}} \left[\nu t^{d-1-\nu}\left(\log t+\frac{1}{\nu}\right)\right]dt
    \end{gather}
    where we have used the bound $h_2(p)=-p\log p-(1-p)(-p-\sum_{k=2}^\infty\frac{p^k}{k})=p-p\log p-\sum_{k=2}^\infty \frac{p^k}{k(k-1)}\leq p-p\log p$, valid for $p \in [0,1]$. Therefore, \eqref{eq:pl_bound_final} converges as $r_0 \to 0$ if $d<\nu$.
    
    We can conclude that for $d \geq \nu$, the initial rate of increase of the conditional entropy is faster than $r_0^d$, but when $d < \nu$, the conditional entropy grows at the same rate as we have seen in the Rayleigh fading and Fermi-Dirac connection functions.
    
    We demonstrate the validity of the result in $\Omega = [0,1]^d$, $d=1,2,3$ in Figure \ref{fig:prop_evidence}, which shows the theoretical estimate along with the true (direct numerical integration) value of $\overline{H}(\graph(r_0, 2))$ given by (\ref{eq:entropy_per_edge}).
    \begin{figure}
        \centering
        \includegraphics[width=\linewidth]{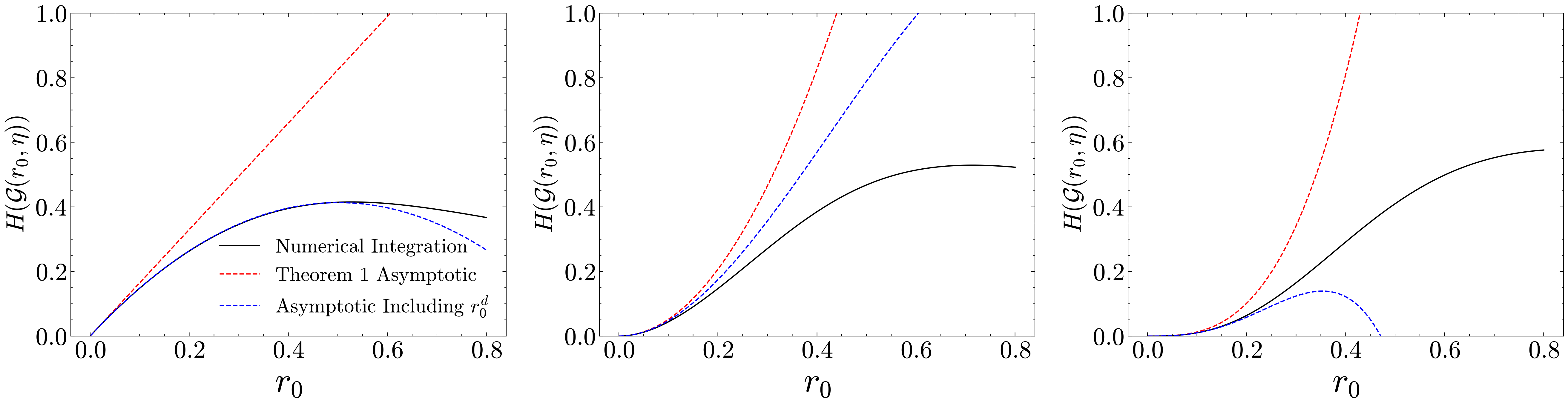}
        \caption{The conditional entropy-per-edge of an SRGG with Rayleigh fading connection function ($\eta=2$), $\overline{H}(\graph(r_0)|\mathcal{R})$, numerically integrated, with the asymptotic from Theorem \ref{small_r0_asymptotic} and equation (\ref{higher_order}) for $\Omega = [0,1]^d$, $d=1,2,3$}
        \label{fig:prop_evidence}
    \end{figure}
    The results in this section show that the limiting behaviour of the conditional entropy as $r_0 \ttz$ is the same for any geometry of the same dimension. Effectively, they say that the small $r_0$ behaviour is determined by the nodes that are far from the boundaries, since the boundary structure of the domain does not influence the limiting entropy. We will explore this idea further in Section \ref{sec:entropy_graph}. To justify the relevance of these results, we note that even for small $r_0$, provided the domain is fixed, the network ensemble is dense as $n\tti$. Figure \ref{fig:prop_evidence} shows that our results are a good estimate of the conditional entropy for approximately $r_0 \in [0, 0.2]$. When $\eta=2$, this corresponds to the typical node connecting to an average of anywhere from around 32\% (in $d=1$) to 3\% (in $d=3$) of the total nodes in the network. So, for short range systems, these results provide practical value. It also appears that these results could be used to derive the scaling as $n\tti$ of \textit{sparse} SRGG ensembles with connection functions of the form $p(r/(s(n)r_0))$, where $s(n) \ttz$ with $n$.

    It is interesting to compare the results of this section to other results about the effect of the ratio $d/\eta$ on other properties of SRGGs with Rayleigh fading connection functions. As an example, in \cite{georgiou2014network}, it is shown that Rayleigh fading improves network connectivity in terms of average degree when $\eta < d$, but deteriorates it when $\eta > d$. In this case, it is clear to see by numerically plotting the asymptotic in the right hand side of Theorem \ref{small_r0_asymptotic} that the conditional entropy of the ensemble for a small, fixed $r_0$ decreases as $\eta$ increases.

    \subsection{Large Connection Range}
    Suppose now that $r_0 \gg D$. The following result shows that there is an explicit dependence on the `moments' of the domain $\Omega$. 
    \begin{theorem}
        \label{thm:large_r0}
        Let $\mathcal{G}(n,r_0,\eta)$ be a SRGG ensemble with a Rayleigh fading connection function on $\Omega$. Then, as $r_0 \tti$,
        \begin{equation}
            \overline{H}(\graph(r_0,\eta)|\mathcal{R}) = \frac{1+\eta \log r_0}{r_0^{\eta}}\mathbb{E}[R^\eta] - \frac{\eta}{r_0^\eta}\mathbb{E}[R^\eta \log R] + \bigO{\frac{1}{r_0^{2\eta}}\log r_0}
        \end{equation}
        where $R$ is a random variable representing the distance between two random points distributed according to $\mu$ in $\Omega$.
    \end{theorem}
    \begin{proof}
        See Appendix \ref{sec:large_r0}.
    \end{proof}
    For some geometries, the expectations $\mathbb{E}[R^{\eta}]$ and $\mathbb{E}[R^\eta \log R]$ are computable exactly, due to the simple form of $f(r)$, for example, when $\Omega$ is the unit line segment, and the unit torus, where $f(r)=2-2r, \,r\in[0,1]$ and $f(r)=2, \, r\in [0,1/2]$ respectively. For more complicated geometries, we can resort to numerical integration when we have a known expression for $f(r)$, or use a Monte-Carlo estimate of the expectations when $f(r)$ is not available. To demonstrate the validity of our result, Figure \ref{fig:large_r0} shows the analytic estimate against the simulated conditional entropy for $[0,1]^d$ for $d=1,2,3$, and $\eta=2$, and $\Omega$ as the unit ball in 2 and 3 dimensions with $\eta= 4$.
    \begin{figure}
        \centering
        \includegraphics[width=\linewidth]{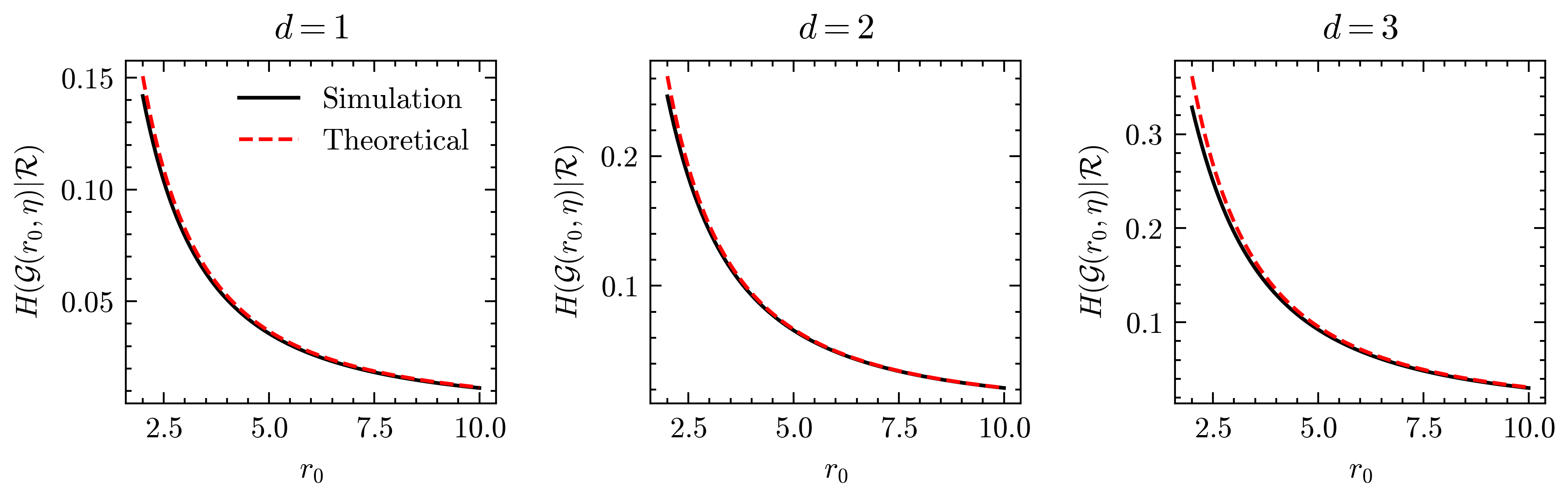}
        \includegraphics[width=0.7\linewidth]{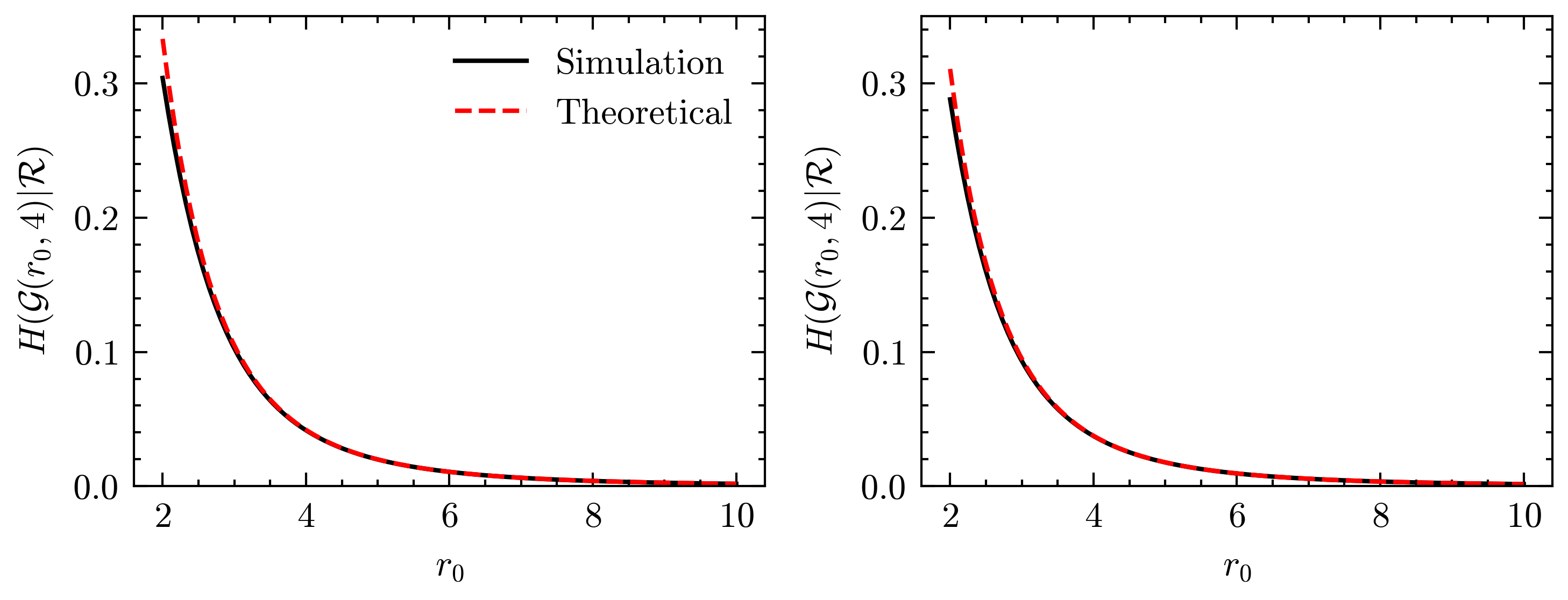}
        \caption{Simulated conditional entropy against the theoretical prediction of Theorem \ref{thm:large_r0} for $\Omega = [0,1]^d$ and $\eta=2$ for $d=1,2,3$ (top left, middle and right respectively), and $\Omega$ as the unit $d$ ball with $\eta=4$ for $d=2, 3$ (bottom left and right respectively).}
        \label{fig:large_r0}
    \end{figure}

    \subsection{Compressibility of SRGGs vs ER Graphs}
    In this section, we will apply the previous two sections in the area of graph compression. It is well known that for any random sequence, the fundamental lower bound for the amount of information required to represent that sequence is the entropy of the distribution that generated it. In our setting, to compress a random network ensemble $\graph$, we require at least $H(\graph)$ bits (or nats since our logarithm is to base $e$) to compress the network. A question we may want to ask is whether the SRGG is any easier to compress than other graph models, for example the Erd\H os-R\'enyi (ER) graph. Of course, the correlation between the edges in the SRGG means it has lower entropy, but can we say anything more? A common quantity in the graph compression literature used to quantify the ease of compressing a graph structure is the `compressibility' of the structure \cite{chierichetti2009models, kontoyiannis2022compression}. 
    \begin{definition}
        Let $\graph_n$ be an ensemble of random graphs, with entropy $H(\graph_n)$, and edge set $E_n$. The compressibility of $\graph_n$ is the ratio
        \begin{equation}
            C_n := \frac{H(\graph_n)}{\mathbb{E}[|E_n|]}
        \end{equation}
        We say an ensemble is compressible if $C_n = \bigO{1}$ as $n\tti$.
    \end{definition}
    It is straightforward to check that the Erd\H os-R\'enyi ensemble is compressible. Indeed,
    \begin{equation}
        C_n^{ER} = \frac{\binom{n}{2}h_2(p)}{\binom{n}{2}p} = \frac{h_2(p)}{p} = \Theta(1)
    \end{equation}
    as $n\tti$. Also, the SRGG is compressible:
    \begin{equation}
        \label{eq:cnsrgg_lb}
        C_n^{SRGG} = \frac{H(\graph_n^{SRGG})}{\binom{n}{2}\bar{p}} \geq \frac{\int_0^D f(r)h_2(p(r))dr}{\bar{p}} = \Theta(1)
    \end{equation}
     as $n\tti$. However, (perhaps counter-intuitively), the \textit{sparse} ER graph is incompressible (see e.g. \cite{kontoyiannis2022compression}). To construct an ER graph with the same expected proportion of links as the SRGG, we simply use $G(n,\bar{p})$. Then we are interested in estimating the quantity
    \begin{equation}
        \Delta C_n := C_n^{ER} - C_n^{SRGG}
    \end{equation}
    The lower bound \eqref{eq:cnsrgg_lb} becomes exact in the large $n$ limit \cite{janson2010graphons}, and $C_n^{ER} = h_2(\bar{p})/\bar{p}$ for every $n$, so we can define the compressibility difference in the large $n$ limit as
    \begin{equation}
        \Delta C := \lim_{n\tti} \Delta C_n = \frac{h_2(\bar{p}) - \int_0^Df(r)h_2(p(r))dr}{\bar{p}} 
    \end{equation}
    Using the asymptotic expressions derived in the previous sections leads to the following result.
    \begin{theorem}
        \label{compressibility_thm}
        Let $\graph$ be an ensemble of SRGGs with a Rayleigh fading connection function with connection range $r_0$, path loss exponent $\eta$, and average connection probability $\bar{p} = \bar{p}(r_0)$. Then the compressibility difference in the large $n$ limit between the SRGG ensemble $\graph$ and $G(n,\bar{p})$, $\Delta C$ satisfies
        \begin{equation}
            \lim_{r_0 \ttz} \Delta C = \infty
        \end{equation}
        \begin{equation}
            \lim_{r_0\tti} \Delta C = 0 
        \end{equation}
    \end{theorem}
    \begin{proof}
        See Appendix \ref{compressibility_proof}.
    \end{proof}
 Intuitively the difference between the small $r_0$ case and the large $r_0$ case is that when $r_0$ is large, the connection function is effectively 1 everywhere, so there is not much difference between the SRGG and the ER graph. However, when $r_0$ is small, the SRGG connection function maintains a very small section where the connection probability is high, and therefore maintains its high clustering which improves the compressibility. We can also say that if the connection function satisfies the integrability condition (\ref{condition}), then the divergence of $\Delta C$ as $r_0\ttz$ will also occur. 

\section{The Entropy Graph}
\label{sec:entropy_graph}
    In this section, we tackle the issue of estimating the conditional entropy-per-edge of the SRGG for general geometries and connection functions. We begin by writing the conditional entropy-per-edge as an expectation over distances.
    \begin{equation}
        \label{entropy_edge_density}
        \overline{H}(\graph(r_0)|\mathcal{R}) = \int_0^D h_2(p(r/r_0))f(r)dr = \mathbb{E}_R[h_2(p(R/r_0))]
    \end{equation}
    Now by analogy with the expression for the expected edge density of the SRGG, $\bar{p} = \int_0^D f(r)p(r)dr$, (\ref{entropy_edge_density}) can be thought of as the expected edge density of a SRGG with connection function $h_2 \circ p$. We will call this graph the `\textit{entropy graph}'. This observation makes the task of estimating $H(\graph(r_0)|\mathcal{R})$ significantly easier, as now instead of requiring the function $f$ and numerically estimating the integral, we can now estimate the conditional entropy by a Monte-Carlo simulation of network edge counts. It also allows us to directly quantify the effect of geometric boundaries on the conditional entropy of the network ensemble.
    
    \subsection{The Effect of Boundaries for Small Connection Ranges}
    The formulation of the conditional entropy-per-edge $H(\graph|\mathcal{R})$ as an edge count of the network allows us to think about the network complexity in a new way. Below, we will discuss the \textit{entropy mass} of a point $x$, which can be thought of as the `amount of uncertainty' a point at $x$ contributes to the overall ensemble entropy.
    \begin{definition}[Entropy Mass]
        For a SRGG ensemble with connection function $p$ on $\Omega$, We define the entropy mass at a point $x \in \Omega$ as
        \begin{equation}
            H_x(\graph|\mathcal{R}) = \int_{\Omega} h_2(p(\|x-y\|))dy
        \end{equation}
    \end{definition}
    Recall that in this section, we will not treat $r_0$ as a variable, and instead treat $r_0$ as a part of the connection function. Hence, we have not included $r_0$ in the definition of entropy mass, and will treat connection functions that have the same functional form but different parameter $r_0$ as different connection functions entirely. It is useful here to draw an analogy with the `connectivity mass' from the SRGG connectivity literature. It is shown by Penrose in \cite{penrose2016connectivity} that the probability of a SRGG with a connection function $p$ subject to some regularity conditions being connected satisfies
    \begin{equation}
        \lim_{n\rightarrow\infty} \left|\prob{G \text{ is connected}} - \exp\left(-n\int_{\Omega} \exp\left(-n\int_{\Omega}p(\|x-y\|)dy\right)dx\right)\right|= 0
    \end{equation}
    Where the \textit{connectivity mass} $M(x) = \int_\Omega p(\|x-y\|)dy$ clearly appears in the expression. We point out that our \textit{entropy mass} is the \textit{connectivity mass} of a SRGG with connection function $h_2\circ p$. Unfortunately, our connection function $h_2 \circ p$ does not satisfy the regularity conditions set out in \cite{penrose2016connectivity} as for example it is non-monotone. However, it is still a useful object to study. In \cite{dettmann2016random}, a general formula for the connectivity mass is given in terms of the boundary components of the domain. The results of the next few sections can be viewed as an application of the methodology of \cite{dettmann2016random} to the problem of graph entropy. This seems to be an unexpected use case for research into connectivity, and prompts the question of whether there are other use cases that this technology can be used for. 
    
    In Figure \ref{fig:localised_entropy_square}, we plot the entropy mass for each point in a unit square $[0,1]^2$, for a connection function $p = p(r/r_0)$ with short ($r_0$ small), medium ($r_0$ maximising the conditional entropy-per-edge) and long range ($r_0 = 1$) connections. We can clearly see that the position of the node greatly impacts its influence on the overall entropy of the ensemble. In particular, in the left panel, (small $r_0$), the entropy mass of the corner nodes is approximately one quarter of that of the bulk nodes. The entropy mass of the node at the centre is given approximately (assuming $h_2(p(1,r_0)) \approx 0$) by integrating the entropy function radially from the centre:
    \begin{equation}
        H_{(1/2,1/2)}(\graph) \approx \int_0^{2\pi}\int_0^1 h_2(p(\|(r\cos\phi,r\sin\phi)\|)) rdrd\phi
    \end{equation}
    \begin{equation}
        = 2\pi \int_0^1 h_2(p(r)) rdr
    \end{equation}
    but a node on the middle of the edge has entropy mass 0 in the direction that points outside of the box, so its region that we need to integrate over is only the semicircle,
    \begin{equation}
        H_{(1/2,0)}(\graph)  \approx \int_0^{\pi}\int_0^1 h_2(p(\|(r\cos\phi,r\sin\phi)\|)) rdrd\phi = \pi\int_0^1 h_2(p(r)) rdrd\phi
    \end{equation}
    and finally by the same reasoning the corner node's mass is
    \begin{equation}
        H_{(1,0)}(\graph) \approx \int_0^{\pi/2}\int_0^1 h_2(p(\|(r\cos\phi,r\sin\phi)\|)) rdrd\phi = \frac{\pi}{2}\int_0^1 h_2(p(r)) rdrd\phi
    \end{equation}
    so we can clearly see the $1:1/2:1/4$ ratio of entropy mass depending on the point's location in the small $r_0$ limit. Again it is useful to think of this like the connectivity mass of the SRGG. As an example, Figures 1 and 2 in \cite{coon2012full} shows similar behaviour. The interesting feature of the entropy mass is that the dominant contribution \textit{changes} from the bulk to the corners to the edges as the average connectivity probability increases.  
    \begin{figure}
        \centering
        \includegraphics[width=0.32\linewidth]{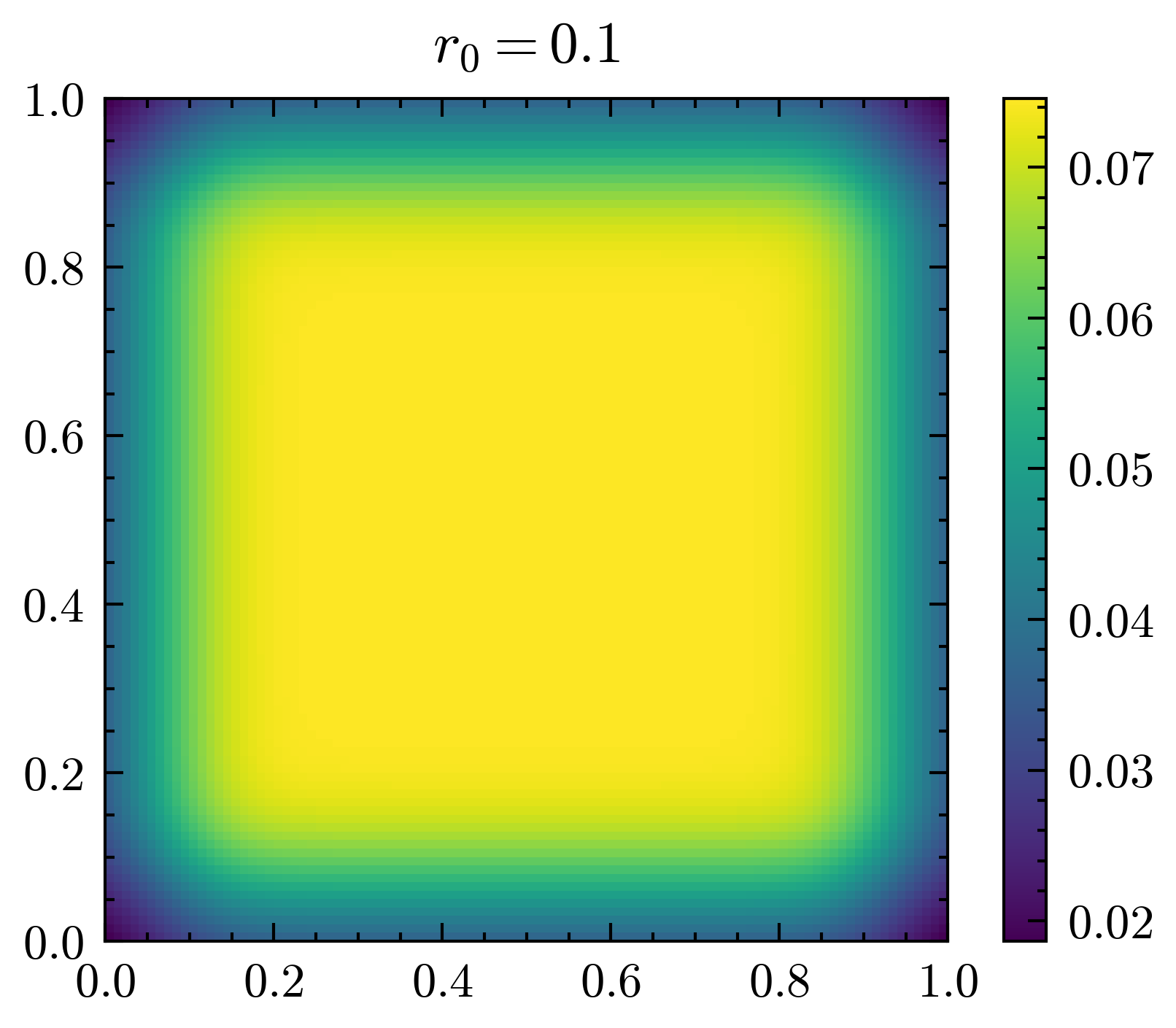}
        \includegraphics[width=0.32\linewidth]{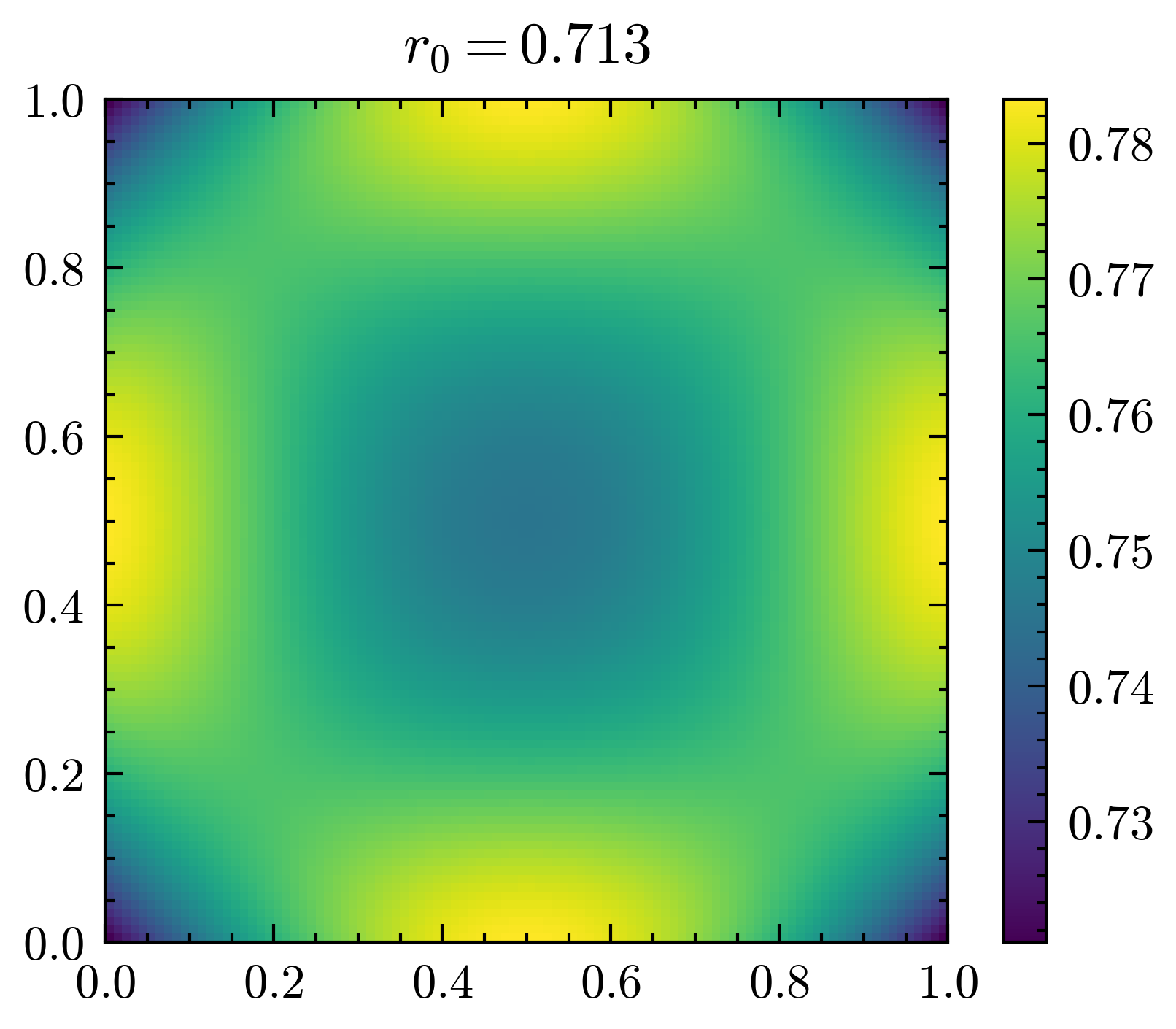}
        \includegraphics[width=0.32\linewidth]{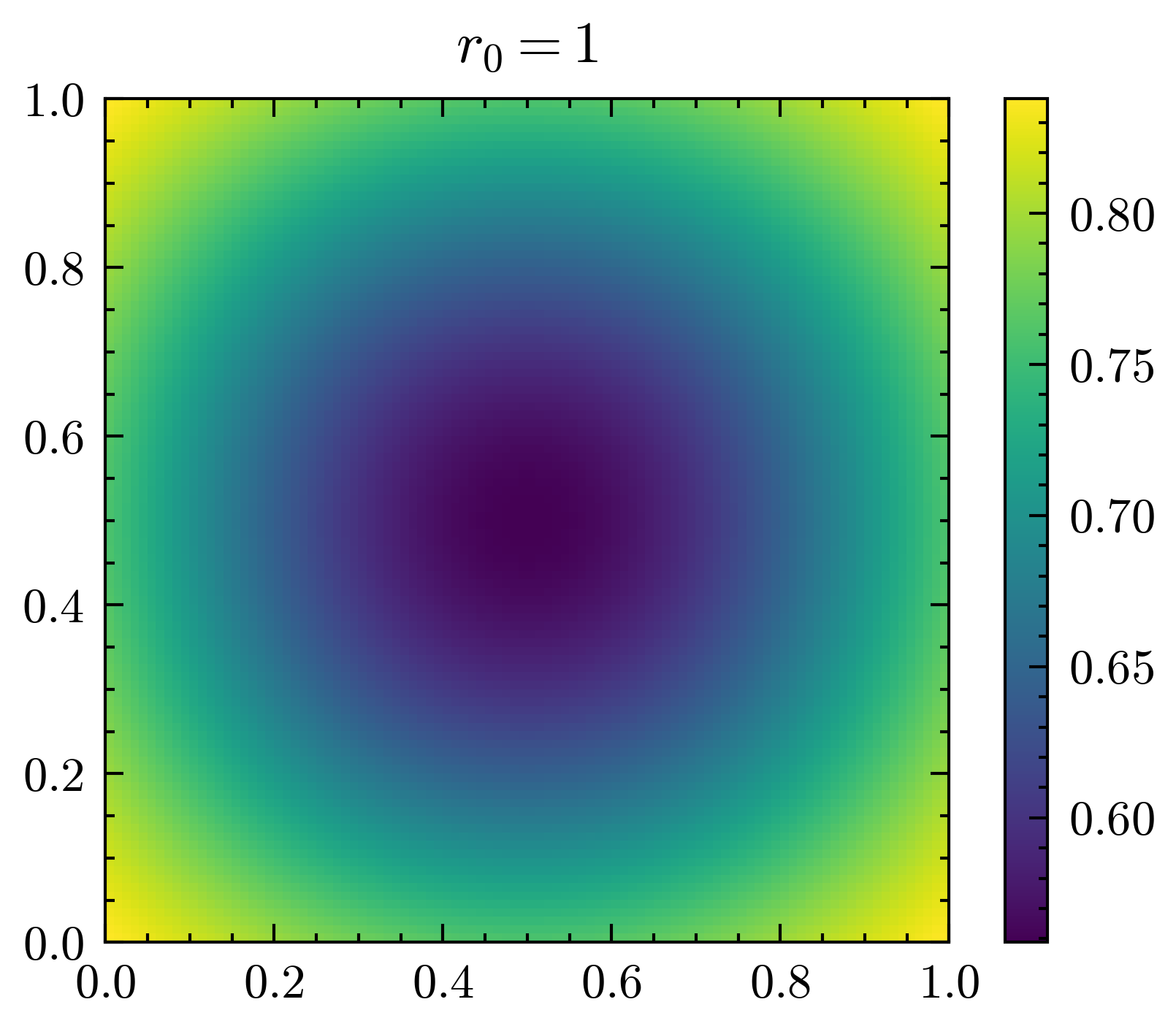}
        \caption{Entropy mass at each point of the unit square. We see that the main contribution to the entropy comes from the bulk, then the edges, then the corners as $r_0$ increases (connection function $p(r/r_0) = \exp(-(r/r_0)^2)$).}
        \label{fig:localised_entropy_square}
    \end{figure}

    In the next section, we aim to quantify the effects of the enclosing geometry on the graph entropy using this idea. We will use two different geometries that each demonstrate different boundary effects.
    
    \subsection{Entropy Estimation for General Geometries}
        Here, we will present two further examples that demonstrate the use of the entropy graph formalism.

        \subsubsection{Wedge - Corner Effects by Connectivity-Style Analysis}
        \label{sec:entropy_mass}
        
        In this section we will follow the methodology of \cite{dettmann2016random} to establish the form of the entropy mass at each point in a wedge. The idea of using a wedge geometry is that it is representative of a general boundary component in a 2D polygonal domain. We define the wedge with angle $\theta$ in polar coordinates as the following set,
        \begin{equation}
            W_{\theta} := \{(r,\phi): 0 \leq \phi \leq \theta, 0 \leq r \leq R\}
        \end{equation}
        In \cite{dettmann2016random}, it is assumed that the connection function $p$ takes the form
        \begin{equation}
            p(r) = p(0) + \sum_{\alpha \in \mathcal{A}} a_{\alpha}r^{\alpha}
        \end{equation}
        for small $r$, where $\alpha, a_{\alpha} \in \mathbb{R}$, and $\alpha > 0$. That is, it has a generalised Taylor expansion for small $r$. However in our setting, if we treat $\rho = h_2 \circ p$ as the connection function, then this is often not satisfied. Instead, our connection function has the following small $r$ expansion.
        \begin{equation}
            \label{eq:general}
            \rho(r) = \rho(0) + \sum_{\alpha \in \mathcal{A}}a_{\alpha}r^{\alpha} + \sum_{\beta \in \mathcal{B}} b_{\beta} r^{\beta}\log r
        \end{equation}
        where again, $b_\beta, \beta \in \mathbb{R}$, and $\beta > 0$. Note that if $p(0) = 1$ then the constant term $\rho(0)$ is 0. We denote by $\alpha_\min$ and $\beta_\min$ the smallest elements of $\mathcal{A}$ and $\mathcal{B}$ respectively. This means that the methods of \cite{dettmann2016random} do not immediately apply, but the general principles still work, with a little extra effort. We assume the following:
        \begin{assumption}
            \begin{enumerate}
                \item The entropy function $\rho$ can be written in the form (\ref{eq:general})
                \item $\rho$ is piecewise smooth with nonsmooth points at a discrete and possibly empty set $\{r_k\}$ for $k = 1,...,K$, where $K \in \mathbb{N}$ is the number of such discontinuities.
                \item The integral $\int_0^\infty r\rho(r)dr$ is finite.
                \item All derivatives of $\rho(r)$ are monotone for sufficiently large $r$.
            \end{enumerate}
        \end{assumption}
        Under this assumption, we will show that the entropy mass for a polar point $(r, \omega)$ in the wedge $W_{\theta}$ situated close to the corner is given by
        \begin{equation}
            \label{eq:entropy_mass_expression}
            H_{(r,\omega)}(\graph) = \theta \rho_1 + \rho_0 r(\sin \omega + \sin \omega') + \rho(0) \frac{r^2}{2}(\sin\omega\cos \omega + \sin \omega'\cos \omega') + \bigO{r^3, r^{\alpha_{\min}+2}, r^{\beta_{\min}+2}\log r} 
        \end{equation}
        where $\omega' = \theta-\omega$, and $\rho_k := \int_0^{\infty} r^k\rho(r)dr$
        The derivation is given in Appendix \ref{sec:wedge}. It is interesting to compare the results to those in \cite{dettmann2016random}. In fact, the form of the entropy mass \eqref{eq:entropy_mass_expression} is exactly the same as the form of the connectivity mass, up to the omitted terms (third order). The first few omitted terms are calculated in the Appendix \ref{sec:corner_mass}, but are very cumbersome so are not reproduced here. It is only at the $\bigO{r^3, r^{\alpha_\min}. r^{\beta_\min}\log r}$ term where we see the first $\log r$ term, which is the first difference between the entropy mass and the connectivity mass.
        \begin{figure}
            \centering
            \includegraphics[width=0.49\linewidth]{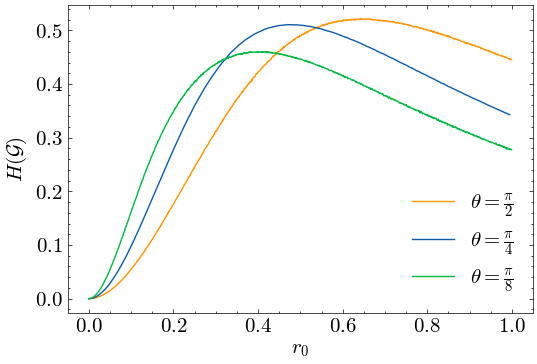}
            \includegraphics[width=0.49\linewidth]{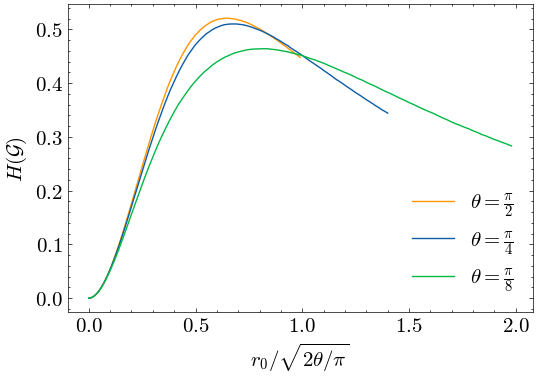}
            \caption{Left: Entropy curves with $\eta = 2$ in sectors of unit radius with angles $\frac{\pi}{2}, \frac{\pi}{4}$ and $\frac{\pi}{8}$, Right: The same curves rescaled by dividing $r_0$ by $\sqrt{2\theta/\pi}$}
            \label{fig:wedge}
        \end{figure}

        \begin{figure}
            \centering
            \includegraphics[width=\linewidth]{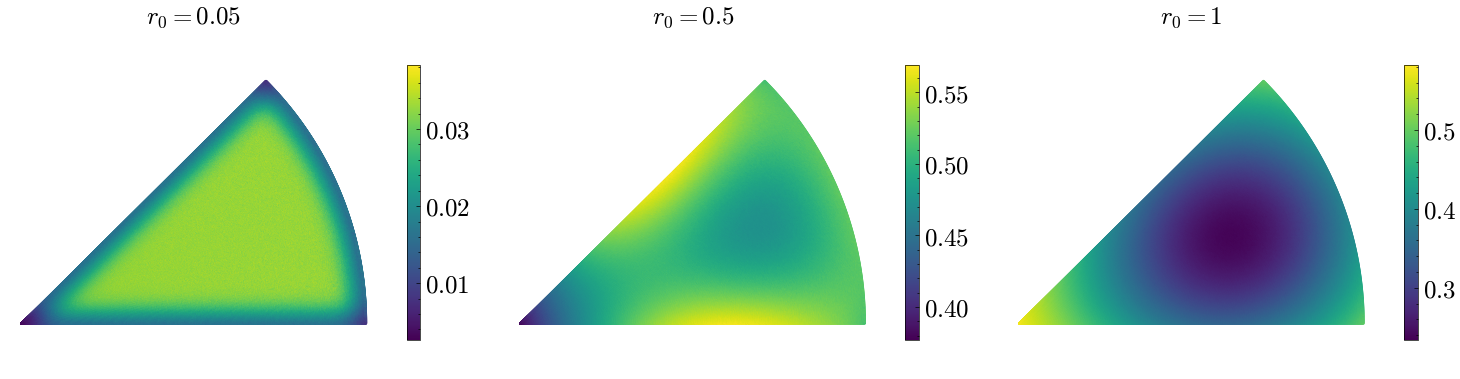}
            \caption{The entropy mass of each point in a wedge of angle $\theta = \pi/4$ for $r_0 = 0.05, 0.5, 1$}
            \label{fig:wedge_mass}
        \end{figure}
        Differentiating this expression with respect to $\omega$ shows that the local optima occur when $\theta = \omega/2$, that is, when the node is most separated from the boundary. Also, the form of (\ref{eq:entropy_mass_expression}) explicitly shows the effect of the angle $\theta$. Taking $\theta = \pi$ and $r = 0$ gives the expression for a node located on an edge component, and explains why the nodes in the corner of Figure \ref{fig:localised_entropy_square} have half the entropy mass of the nodes at the corners ($\theta = \pi/2$, $r=0$), and the corner nodes in panel 1 of Figure \ref{fig:wedge_mass} have one quarter of the entropy mass of the edge nodes (here, $\theta = \pi/4$, $r=0$). It is interesting to see that, again, the dominant contribution to the entropy in Figure \ref{fig:wedge_mass} comes from the bulk, edges and corners in turn as $r_0$ increases. It is likely that this behaviour can be explained by similar methods to the expansion given in equation \eqref{eq:entropy_mass_expression} in different regimes (e.g. $r/R$ approximately constant). 
        
        % In principal, the behaviour in the larger $r_0$ regimes (panel 2 and 3 of Figure \ref{fig:wedge_mass}) could be explained by the higher order terms in the expansions in Appendix \ref{sec:wedge}, for example by taking $r \sim r_0$ for panel 2, and $r \ll r_0$ for panel 3.
        
         We emphasise that the expressions given here and in Appendix \ref{sec:wedge} may be used to estimate the conditional entropy in very general geometries by decomposing them into boundary components, including extension to three dimensions as is done in \cite{dettmann2016random} for the connectivity probability.

\subsubsection{Cantor Set - Complicated Geometry with Unavailable $f(r)$}
    \label{cantor}

In this section, we consider an SRGG embedded in the Cantor set (or equivalently, an SRGG in $[0,1]$ with Cantor-distributed nodes). We define the Cantor set with ratio $\lambda \in \mathbb{R}$ with $\lambda > 2$ iteratively as follows. We start with $\mathcal{C}_0 := [0,1]$. Then, we remove all but the leftmost and rightmost $\lambda$-ths of the set, so that 
\begin{equation}
    \mathcal{C}_1 := [0,1/\lambda] \cup[1-1/\lambda, 1]
\end{equation}
In general, defining $L(x) = \frac{1}{\lambda}x$ and $R(x) = 1-\frac{1}{\lambda} + \frac{1}{\lambda}x$, we have $\mathcal{C}_k = L(\mathcal{C}_{k-1}) \cup R(\mathcal{C}_{k-1})$, and $\mathcal{C} = \bigcap_{k=1}^{\infty} \mathcal{C}_k$. In Figure \ref{fig:cantor_entropy}, the entropy curves for this choice of $\Omega$ are shown. This is an example of when the entropy graph formalism is useful for estimating conditional entropy when we do not have $f(r)$. This is also an interesting counterexample to the intuition that the entropy curves should be unimodal. The binary entropy function $h_2$ is unimodal, but this does not mean the entropy of the SRGG is.

The goal here is to explain why the conditional entropy of the SRGG in the $\mathcal{C}_k$ for a large $k$ displays these log-periodic oscillations, with a `gradient' of the Hausdorff dimension of the Cantor set, which is $\log 2 / \log (\lambda)$ \cite{falconer1985geometry}. It is important to note that the usual definition of the uniform distribution, that is, taking the normalised Lebesgue measure, will not work here, since the Cantor set is a set of zero Lebesgue measure. Therefore, we will define the uniform distribution on the Cantor set as the normalised $d$-dimensional Hausdorff measure on $\mathcal{C}$, where here $d = \log(2)/\log(\lambda)$ \cite{falconer1985geometry}. For convenience we will denote by $F(r)$ now the CDF of the displacement between two random points, that is, if $X, Y$ are random points in $\mathcal{C}$, then $F$ is the distribution of $R := X-Y$.
\begin{figure}
    \centering
    \includegraphics[width=0.98\linewidth]{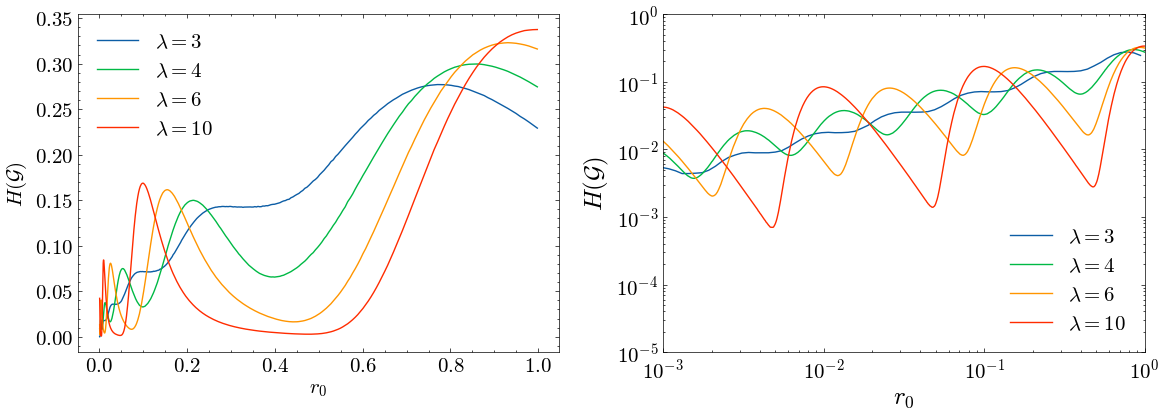}
    \caption{Left: The entropy curve of a SRGG with Rayleigh connection functions with $\eta=4$ for $\lambda = 3, 4, 6, 10$. Right: The same curves on a log-log plot. The gradient of the line connecting the local maxima in the log-log plot is $\dim_H(\mathcal{C})$.}
    \label{fig:cantor_entropy}
\end{figure}

Our aim is to explain the log-periodic oscillatory behaviour of the entropy curve in Figure \ref{fig:cantor_entropy}. Provided that the connection function satisfies some decay and differentiability conditions (Assumption \ref{eq:moment_bounded} in Appendix \ref{sec:residue}), we will show that the conditional entropy-per-edge $\overline{H}(\graph(r_0)|\mathcal{R})$ may be written as
\begin{equation}
    \overline{H}(\graph(r_0)|\mathcal{R}) \sim r_0^d\left(R_0 + 2\sum_{m=1}^{\infty} R_m \cos\left(\theta_m + 2\pi m \frac{\log r_0}{\log{\lambda}}\right)\right)
\end{equation}
as $r_0\ttz$ where for $m \in \mathbb{N}_0$, $R_m, \theta_m \in \mathbb{R}$ are constants to be determined. We begin with the function
\begin{equation}
    \overline{H}(\graph(r_0)|\mathcal{R}) = \int_{-1}^{1} h_2(p(|r|/r_0))dF(r) 
\end{equation}
\begin{equation}
    = 2\int_0^1 \rho(r) dF(r)
\end{equation}
where we have introduced the function $\rho(r) = h_2(p(|r|/r_0))$. For $s \in \mathbb{C}$, let $\mathcal{M}[\rho; s] = \int_0^{\infty} \rho(r)r^{s-1}dr$ be the Mellin transform of $\rho$, and for $s \in \mathbb{C}$, define the Cantor moment $C[F;s] := \int_0^1 r^{s}dF(r)$. The next results will allow us to swap the order of integration. They may also be of general use in the study of distance distributions in the Cantor set.

\begin{lemma}
    \label{thm:recursive}
    Let $r \in \mathbb[0,1]$, then
    \begin{equation}
        F(r) = \frac{1}{2}F(\lambda r) + \frac{1}{4}F(\lambda r-(\lambda - 1)) + \frac{1}{4}F(\lambda r+(\lambda-1))
    \end{equation}
\end{lemma}
\begin{proof}
See Appendix \ref{lemma1}.
\end{proof}

\begin{lemma}
    \label{thm:moments}
    Let $s \in \mathbb{C}$ with $\Im(s) \neq \frac{2\pi i m}{\log \lambda}$ for any $m \in \mathbb{Z}$. The Cantor distance moments $C[F;s]$ are given by
    \begin{equation}
        C[F;-s] = \frac{1}{(2\lambda^{-s}-1)}\sum_{l=0}^{\infty}\binom{-s}{2l}(\lambda-1)^{-s-2l}C[F;2l]
    \end{equation}
    where for $z \in \mathbb{C}$ and $k \in \mathbb{Z}$, $\binom{z}{k} = \frac{(z)_l}{k!}$, with $(z)_k = z\cdot(z-1)\cdot...\cdot(z-k+1)$ as the P\"ochhammer symbol. Further $C[F;-s]$ is a meromorphic function. 
\end{lemma}

\begin{proof}
    See Appendix \ref{lemma3}.
\end{proof}

By Lemma \ref{thm:moments} $C[F;-s]$ has simple poles at $-s_m = -\frac{\log 2}{\log \lambda} - \frac{2\pi im}{\log \lambda}$ for every $m \in \mathbb{Z}$. In particular, this means it is analytic on the half-planes $\Re(s) < \frac{\log 2}{\log \lambda}$ and $\Re(s) > \frac{\log 2}{\log \lambda}$. Using a change of variables $t = r/r_0$, it is easy to see that
\begin{equation}
    \mathcal{M}[\rho; s] = r_0^s\psi(s)
\end{equation}
where $\psi(s) = \int_0^{\infty} t^{s}h_2(p(t))dt$, which we have assumed to be integrable for $\Re(s) > 0$ (Assumption \ref{eq:moment_bounded}). So, for $\Re(s)>0$ we can use the Cauchy-Riemann equations to show that $\psi(s)$ is also analytic.
Then, $\mathcal{M}[\rho;s]$ and $C[F;-s]$ have a common strip of analyticity, and therefore Parseval's theorem for the Mellin transform allows us to write
\begin{equation}
    \label{eq:integral_to_compute}
    \overline{H}(\graph(r_0)|\mathcal{R}) = \frac{1}{\pi i}\int_{c_l-i\infty}^{c_l+i\infty}M[\rho;s]C[F;-s]ds
\end{equation}
where $c_l$ is a real number in the interval $(0, \log(2)/\log(\lambda))$. We may use the residue theorem on a rectangular contour surrounding the poles $s_m$. The details are technical, and so are relegated to Appendix \ref{sec:residue}.

\begin{theorem}
    \label{thm:cantor_entropy_expression}
    The conditional entropy-per-edge of the SRGG with a connection function satisfying Assumption \ref{eq:moment_bounded} in the Cantor set is given by
    \begin{equation}
        \overline{H}(\graph(r_0)|\mathcal{R}) = r_0^{d}\left(R_0 +2\sum_{m=1}^{\infty}R_m\cos\left(\theta_m + 2\pi m\frac{\log r_0}{\log \lambda} \right)\right) + o(r_0^d)
    \end{equation}
    as $r_0\ttz$, where $d = \log(2)/\log(\lambda)$, and $R_m$ and $\theta_m$ are the modulus and argument of \begin{equation*}
    g(m)= \frac{\psi(s_m)}{\log \lambda}\int_0^1\left[(r+\lambda-1)^{-s_m}+(-r+\lambda-1)^{-s_m}\right]dF(r)
    \end{equation*} 
    respectively.
\end{theorem}
\begin{proof}
    See Appendix \ref{sec:residue}.
\end{proof}

We can make a number of observations. It can be clearly seen that the leading-order behaviour of the conditional entropy as $r_0 \ttz$ is a log-periodic oscillation. Also, this result clearly mirrors Theorem \ref{small_r0_asymptotic}, since the leading order behaviour is $r_0^d$, except here $d$ is the Hausdorff dimension of the Cantor set.

\section{Conclusion}
\label{sec:conclusion}
In this paper we have studied the entropy of a family of soft random geometric graph ensembles in the limiting cases of a very small and very large connection range. In Section \ref{sec:scaling} we derived the asymptotic scaling of the conditional entropy for SRGGs with Rayleigh fading connection functions. We showed that in the small $r_0$ limit the conditional entropy scales as $\Theta(r_0^d)$ for general connection functions provided they satisfy an integrability condition. It is very likely that this can be used to derive the limiting conditional entropy for sparse ensembles whose connection ranges vary with $n$. We then used these results to show that the difference in the amount of conditional entropy-per-edge between ER graphs and SRGGs diverges when the probability of a connection gets small, but goes to 0 for large connection ranges. This suggests that algorithms to compress SRGGs should take advantage of the clustering present in the networks. \\

In Section \ref{sec:entropy_graph} we introduced the `entropy graph' as an object to study SRGG entropy. The quantity $H(\graph(r_0)|\mathcal{R})$ may be written as an expectation over distances which corresponds to an edge-level property of the SRGG. This allows us to approach the problem of estimating the conditional entropy of the SRGG as an average degree problem. We demonstrated that this formalism is useful when $f(r)$ is unavailable due to the complexity of the underlying geometry by characterising the conditional entropy in a general boundary component, and in a fractal domain. Potential applications of this idea include designing sensor networks in complicated domains where we want to measure the uncertainty of the global network structure. In particular, it suggests that in a random spatial network, there is location-induced heterogeneity in the amount of information each node requires to represent. Finally, we note that this idea opens up the field of spatial network entropy to approaches involving stochastic geometry, and motivates the study of soft random geometric graphs with non-monotone connection functions. 

\section*{Acknowledgements}

The authors would like to thank an anonymous referee, whose feedback has greatly improved this paper. OB acknowledges funding from the EPSRC Centre for Doctoral Training in Computational Statistics and Data Science (COMPASS), Grant Number EP/S023569/1. This work made use of BluePebble, one of the University of Bristol’s high-performance computing facilities \url{https://www.bristol.ac.uk/acrc/high-performance-computing/}.

\section*{Data Availability}

Data and code used to produce the figures in this article are available on GitHub: \url{https://github.com/oijbaker/SRGG-Entropy}.

\section*{Conflict of Interest}
The authors have no conflict of interest to declare.

\appendix
\section*{Appendix}
\section{Section \ref{sec:scaling} Proofs}
\label{sec:proofs}

\subsection{Proof of Theorem \ref{small_r0_asymptotic}}
\label{small_r0_proof}
    To prove this theorem, we will require the following Lemmas.
    \begin{lemma}[\cite{Topsøe2001} Theorem 1.2]
        \label{thm:entroyp_sandwich}
        Let $x \in [0,1]$, then
        \begin{equation}
            (4\log 2) x(1-x) \leq h_2(x) \leq (e \log 2) (x(1-x))^{\frac{1}{\log(4)}}
        \end{equation}
    \end{lemma}
    \begin{lemma}
        \label{thm:int_0_D}
        Let $m \geq 0$, $c > 0$, $D > 0$, $\eta > 0$, then
        \begin{equation}
            \int_0^D r^m \exp\left(-c\left(\frac{r}{r_0}\right)^{\eta}\right) dr = \frac{1}{\eta}r_0^{m+1} c^{-\frac{(m+1)}{\eta}} \gamma\left(\frac{m+1}{\eta}, c\left(\frac{D}{r_0}\right)^\eta\right)
        \end{equation}
        where $\gamma(z, x) = \int_0^x t^{z-1}e^{-t}dt$ is the lower incomplete gamma function.
    \end{lemma}
    \begin{proof}
        This can be seen by directly integrating the left hand side with the change of variables $t = c\left(\frac{r}{r_0}\right)^{\eta}$.
    \end{proof}
    \begin{proof}[Proof of Theorem \ref{small_r0_asymptotic}]
    
        The aim is to compute the integral
        \begin{equation}
            \overline{H}(\graph(r_0,\eta)|\mathcal{R}) = \int_0^D f(r)h_2(p(r/r_0))dr
        \end{equation}
        for small $r_0$.
        We can split the integral into a contribution from `small' and `large' distances,
        \begin{equation}
            \int_0^D f(r)h_2(p(r/r_0))dr = \int_0^{\sqrt{r_0}} f(r)h_2(p(r/r_0))dr + \int_{\sqrt{r_0}}^D f(r)h_2(p(r/r_0))dr
        \end{equation}
        For each $r \in (0,D]$, $p(r/r_0) \ttz$ as $r_0 \ttz$. To show that the second integral converges to 0 faster than $o(r_0^d)$, we write
        \begin{equation}
            \int_{\sqrt{r_0}}^D f(r)h_2(p(r/r_0))dr \leq \int_{\sqrt{r_0}}^Df(r)dr \cdot \sup_{\sqrt{r_0} < r < D} h_2(p(r/r_0))    
        \end{equation}
        \begin{gather}
            = \prob{R > \sqrt{r_0}} \cdot \sup_{\sqrt{r_0} < r < D} h_2(p(r/r_0)) \nonumber \\
            \leq \sup_{\sqrt{r_0} < r < D} e\log(2)(e^{-(r/r_0)^{\eta}}(1-e^{-(r/r_0)^{\eta}}))^{\frac{1}{\log 4}} \nonumber \\
             = e \log(2) e^{-(1/\sqrt{r_0})^{\eta}}(1-e^{-(1/\sqrt{r_0})^{\eta}}))^{\frac{1}{\log 4}} 
        \end{gather}
        where we have used Lemma \ref{thm:entroyp_sandwich}, and that the probability must be less than 1. For small enough $r_0$, the supremum is achieved at $r=\sqrt{r_0}$
        \begin{gather}
              \sup_{\sqrt{r_0} < r < D} e\log(2)(e^{-(r/r_0)^{\eta}}(1-e^{-(r/r_0)^{\eta}}))^{\frac{1}{\log 4}} = e \log(2) (e^{-(1/\sqrt{r_0})^{\eta}}(1-e^{-(1/\sqrt{r_0})^{\eta}}))^{\frac{1}{\log 4}} \nonumber \\
              \leq e\log (2) e^{-r_0^{-\eta/(2\log(4))}}
        \end{gather}
        which is exponentially decreasing as $r_0 \ttz$, and therefore is clearly $o(r_0^d)$. This means that the asymptotics of $\overline{H}(\graph(r_0,\eta)|\mathcal{R})$ are the same as the integral $\int_0^{\sqrt{r_0}}f(r)h_2(p(r/r_0))dr$,  so it remains to compute the limit of $\int_0^{\sqrt{r_0}} f(r)h_2(p(r/r_0))dr$ as $r_0 \ttz$. The integrand can be split up into two terms,
        \begin{equation}
            f(r)h_2(p(r/r_0)) = -f(r)p(r/r_0)\log p(r/r_0) - f(r)(1-p(r/r_0))\log (1-p(r/r_0))
        \end{equation}
        and we will deal with the two separate terms in sequence. Integrating both sides of the above, the first term under the specific Rayleigh fading model is
        \begin{equation}
            \int_0^{\sqrt{r_0}}f(r)\left(\frac{r}{r_0}\right)^{\eta}e^{-\left(\frac{r}{r_0}\right)^\eta}dr
        \end{equation}
        If our domain $\Omega$ satisfies Assumption \ref{assumption_omega}, then we have that $f(r) = s_{d-1}r^{d-1}(1+o(r))$ as $r \ttz$. This comes from the analytical treatment in Section 2 of \cite{guinier1955small} which verifies this expansion for geometries satisfying the given assumptions. Higher order terms are available, although are more complicated and require additional assumptions \cite{ciccariello1995integral}. So,
        \begin{equation}
            \int_0^{\sqrt{r_0}}f(r)\left(\frac{r}{r_0}\right)^{\eta}e^{-\left(\frac{r}{r_0}\right)^\eta}dr = \frac{s_{d-1}}{r_0^{\eta}}\int_0^{\sqrt{r_0}}r^{d+\eta-1}e^{-\left(\frac{r}{r_0}\right)^\eta}(1+o(r))dr
        \end{equation}
        which may be easily integrated (by Lemma \ref{thm:int_0_D})
        \begin{equation}
            \int_0^{\sqrt{r_0}}f(r)\left(\frac{r}{r_0}\right)^{\eta}e^{-\left(\frac{r}{r_0}\right)^\eta}dr = \frac{s_{d-1}}{\eta}r_0^d\gamma\left(\frac{d}{\eta}+1, r_0^{-\frac{\eta}{2}}\right) + o(r_0^{d+1})
        \end{equation}
        \begin{equation}
        \label{term1}
            \sim \frac{s_{d-1}}{\eta}r_0^d\Gamma\left(\frac{d}{\eta}+1\right)
        \end{equation}
        as $r_0 \ttz$, which gives us the first term in the theorem. For the second term, we will use a Taylor expansion and justify that we can swap the order of summation and integration. We have,
        \begin{equation}
            -(1-p(r/r_0))\log (1-p(r/r_0)) = (1-p(r/r_0))(p(r/r_0) + \frac{1}{2}p(r/r_0)^2 + ...)
        \end{equation}
        \begin{equation}
            = p(r/r_0) - \sum_{k=2}^{\infty}\left(\frac{1}{k-1}-\frac{1}{k}\right)p(r/r_0)^k
        \end{equation}
        \begin{equation}
            = p(r/r_0) - \lim_{K \tti} \sum_{k=2}^{K}\left(\frac{1}{k-1}-\frac{1}{k}\right)p(r/r_0)^k
        \end{equation}
        Now we can see that
        \begin{equation}
            \sum_{k=2}^K\left|\left(\frac{1}{k-1}-\frac{1}{k}\right)p(r/r_0)^k\right| \leq \sum_{k=2}^K\left(\frac{1}{k-1}-\frac{1}{k}\right)
        \end{equation}
        \begin{equation}
            = 1 - \frac{1}{K} \leq 1
        \end{equation}
        for every $K$. Therefore, the sequence of partial sums is bounded and converges absolutely, and we are able to perform a switch of summation and integration:
        \begin{equation}
            -\int_0^{\sqrt{r_0}} f(r)(1-p(r/r_0))\log (1-p(r/r_0)) dr = \int_0^{\sqrt{r_0}} f(r)\left[p(r/r_0) - \sum_{k=2}^{\infty}\left(\frac{1}{k-1}-\frac{1}{k}\right)p(r/r_0)^k\right]dr
        \end{equation}
        \begin{equation}
            = \int_0^{\sqrt{r_0}} f(r)p(r/r_0)dr - \sum_{k=2}^{\infty}\left(\frac{1}{k-1}-\frac{1}{k}\right)\cdot \int_0^{\sqrt{r_0}} f(r)p(r/r_0)^kdr
        \end{equation}
        \begin{equation}
            \label{eq:o(r)_is_here}
            = s_{d-1}\int_0^{\sqrt{r_0}} r^{d-1} p(r/r_0) ( 1+ o(r))dr - s_{d-1}\sum_{k=2}^{\infty} \left(\frac{1}{k-1}-\frac{1}{k}\right)\int_0^{\sqrt{r_0}}r^{d-1}p(r/r_0)^k(1+o(r))dr
        \end{equation}
        \begin{equation}
            \label{penultimate}
            \sim \frac{s_{d-1}}{\eta}r_0^d \gamma\left(\frac{d}{\eta}, r_0^{-\frac{\eta}{2}}\right) - \frac{s_{d-1}}{\eta}r_0^d \sum_{k=2}^{\infty}\left(\frac{1}{k-1}-\frac{1}{k}\right)k^{-\frac{d}{\eta}}\gamma\left(\frac{d}{\eta}, kr_0^{-\frac{\eta}{2}}\right)
        \end{equation}
        \begin{equation}
            \label{final_line}
           \sim \frac{s_{d-1}}{\eta}r_0^d \Gamma\left(\frac{d}{\eta}\right)\left(\zeta\left(\frac{d}{\eta}+1\right) -\sum_{k=2}^{\infty}\frac{k^{-\frac{d}\eta}}{k-1}\right)
        \end{equation}
        where in moving from \eqref{eq:o(r)_is_here} to \eqref{penultimate}, we note that the $o(r)$ term in \eqref{eq:o(r)_is_here} contributes is at most 
        \begin{gather*}
            s_{d-1}\int_0^{\sqrt{r_0}} o(r^d) \left[p(r/r_0) - \sum_{k=2}^\infty \left(\frac{1}{k-1}-\frac{1}{k}\right)p(r/r_0)^k\right]dr \nonumber \\ \leq o\left(\int_0^{\sqrt{r_0}} p(r/r_0)r^d dr \right) = o\left(\frac{1}{\eta}r_0^{d+1} \gamma \left(\frac{d+1}{\eta},r_0^{-\eta/2}\right)\right)
        \end{gather*}
        where we have used Lemma \ref{thm:int_0_D}. So, the $o(r)$ term under the integral in \eqref{eq:o(r)_is_here} contributes at most $o(r_0^{d+1})$ to the whole expression, and is therefore sub-leading order. Combining this with (\ref{term1}) gives the result in the statement of the theorem. 
    \end{proof}

\subsection{Proof of Theorem \ref{thm:large_r0}}
\label{sec:large_r0}
\begin{proof}
    We will derive this result using Taylor expansions. We need to find a large $r_0$ expansion of \begin{equation}
        h_2(p(r/r_0)) = -p(r/r_0)\log p(r/r_0) - (1-p(r/r_0))\log (1-p(r/r_0))
    \end{equation}
    which is equivalent to a small $x$ expansion of $h_2(p(x))$ by setting $x=r/r_0$. For small $x$, $\log(1-e^{-x^\eta}) = \eta\log x - \frac{1}{2}x^\eta + \bigO{x^{2\eta}}$. Therefore,
    \begin{equation}
        h_2(p(x)) = -e^{-x^{\eta}}\log(e^{-x^{\eta}}) - (1-e^{-x^{\eta}})\log(1-e^{-x^{\eta}})
    \end{equation}
    \begin{equation}
        = x^{\eta}\exp(-x^\eta)-(1-\exp(-x^\eta))\left(\eta\log x - \frac{1}{2}x^\eta + \bigO{x^{2\eta}})\right)
    \end{equation}
    \begin{equation}
        = x^{\eta}\exp(-x^\eta) - \left(x^\eta -\frac{1}{2}x^{2\eta}+\bigO{x^{3\eta}}\right)\left(\eta\log x - \frac{1}{2}x^\eta + \bigO{x^{2\eta}}\right)
    \end{equation}
    \begin{equation}
        = x^\eta-x^{2\eta} - \left(x^\eta -\frac{1}{2}x^{2\eta}+\bigO{x^{3\eta}}\right)\left(\eta\log x - \frac{1}{2}x^\eta + \bigO{x^{2\eta}}\right) + \bigO{x^{3\eta}}
    \end{equation}
    \begin{equation}
        = x^\eta-x^{2\eta} -\eta x^\eta \log x + \frac{1}{2}x^{2\eta} + \frac{\eta}{2} x^{2\eta} \log x + \bigO{x^{3\eta}}
    \end{equation}
    \begin{equation}
        = x^{\eta} - \eta x^\eta \log(x) + \bigO{x^{2\eta}\log x}
    \end{equation}
    Therefore,
    \begin{equation}
        \overline{H}(\graph(r_0,\eta)|\mathcal{R})  = \int_0^D f(r)h_2(p(r/r_0))dr
    \end{equation}
    \begin{equation}
        = \int_0^D \left(\left(\frac{r}{r_0}\right)^\eta - \eta \left(\frac{r}{r_0}\right)^{\eta}\log \left(\frac{r}{r_0}\right) + \bigO{\left(\frac{r}{r_0}\right)^{2\eta}\log \left(\frac{r}{r_0}\right)} \right)f(r)dr
    \end{equation}
    \begin{equation}
        = \int_0^D \left(\left(\frac{r}{r_0}\right)^{\eta} + \eta\left(\frac{r}{r_0}\right)^{\eta}\log(r_0) - \eta\left(\frac{r}{r_0}\right)^{\eta}\log(r) + \bigO{\left(\frac{r}{r_0}\right)^{2\eta}\log \left(\frac{r}{r_0}\right)}\right) f(r)dr
    \end{equation}
    \begin{equation}
        = \frac{1+\eta\log(r_0)}{r_0^{\eta}}\mathbb{E}[R^{\eta}] - \frac{\eta}{r_0^{\eta}}\mathbb{E}[R^\eta \log R] + \bigO{\frac{1}{r_0^{2\eta}}\log r_0}
    \end{equation}
    as required.
\end{proof}

\subsection{Proof of Theorem \ref{compressibility_thm}}
\begin{proof}
        We need an asymptotic on the value of $\bar{p}$ for small and large $r_0$ to compute this quantity. Following the same reasoning as in Theorem 1, we get for $r_0 \ttz$:
    \begin{equation}
        \bar{p} \sim \int_0^{\sqrt{r_0}}(1+o(1))s_{d-1}r^{d-1}p(r/r_0)dr + o(r_0^d)
    \end{equation}
    \begin{equation}
        = \frac{s_{d-1}}{\eta}r_0^d \Gamma\left(\frac{d}{\eta}\right) + o(r_0^d)
    \end{equation}
    This implies
    \begin{gather}
        h_2(\bar{p}) \sim -\frac{s_{d-1}}{\eta}r_0^d\Gamma\left(\frac{d}{\eta}\right)\log\left(\frac{s_{d-1}}{\eta}r_0^d\Gamma\left(\frac{d}{\eta}\right)\right) \nonumber \\
        -\left(1-\frac{s_{d-1}}{\eta}r_0^d\Gamma\left(\frac{d}{\eta}\right)\right)\log\left(1-\frac{s_{d-1}}{\eta}r_0^d\Gamma\left(\frac{d}{\eta}\right)\right)
    \end{gather}
    \begin{gather}
        = -\frac{s_{d-1}}{\eta}r_0^d\Gamma\left(\frac{d}{\eta}\right)\left[\log(r_0^d) +\log\left(\frac{s_{d-1}}{\eta}\Gamma\left(\frac{d}{\eta}\right)\right)\right] \nonumber \\
        - \left(1-\frac{s_{d-1}}{\eta}r_0^d\Gamma\left(\frac{d}{\eta}\right)\right)\left(-\frac{s_{d-1}}{\eta}r_0^d\Gamma\left(\frac{d}{\eta}\right) + \bigO{r_0^{2d}}\right)
    \end{gather}
    \begin{equation}
        \sim \frac{s_{d-1}}{\eta}r_0^d\Gamma\left(\frac{d}{\eta}\right)\left(\log(r_0^{-d}) + 1 - \log\left(\frac{s_{d-1}}{\eta}\Gamma\left(\frac{d}{\eta}\right)\right)\right)
    \end{equation}
    \begin{equation}
        = \Theta(r_0^d\log(r_0^{-d}))
    \end{equation}
    as $r_0 \rightarrow 0$. Therefore, we know that for small connection ranges the conditional entropy of the ER graph grows faster than the SRGG by a factor of $\log(r_0^{-d})$. Plugging this back into the definition of $\Delta C_n$, we get
    \begin{equation}
        \Delta C_n \sim \frac{\frac{s_{d-1}}{\eta}r_0^d\Gamma\left(\frac{d}{\eta}\right)\left(\log(r_0^{-d}) + 1 - \log\left(\frac{s_{d-1}}{\eta}\Gamma\left(\frac{d}{\eta}\right)\right)\right) - \int_0^Df(r)h_2(p(r))dr}{\frac{s_{d-1}}{\eta}r_0^d\Gamma\left(\frac{d}{\eta}\right)}
    \end{equation}
    Theorem 1 tells us that $\int_0^D f(r)h_2(p(r))dr = \Theta(r_0^d)$ as $r_0 \ttz$, and so
    \begin{equation}
        \Delta C_n = \Theta(\log r_0) 
    \end{equation}
     as $r_0 \ttz$. This means the difference in compressibility diverges in the low edge density limit.
     For large connection ranges,
    \begin{equation}
    \bar{p} = \int_0^Df(r)\exp(-(r/r_0)^{\eta})dr
\end{equation}
\begin{equation}
    = \int_0^Df(r)(1-(r/r_0)^{\eta}+\bigO{(r/r_0)^{2\eta}})dr
\end{equation}
\begin{equation}
    = 1 - \frac{1}{r_0^{\eta}}\mathbb{E}[R^{\eta}] + \bigO{\frac{1}{r_0^{2\eta}}}
\end{equation}
Then after using the Taylor expansion of $\log(1-x)$, we end up with (in a similar way to the large $r_0$ analysis in Theorem 2),
\begin{equation}
    h_2(\bar{p}) = \frac{1}{r_0^{\eta}}\mathbb{E}[R^{\eta}] - \frac{1}{r_0^{\eta}}\mathbb{E}[R^{\eta}]\log\left(\frac{1}{r_0^{\eta}}\mathbb{E}[R^{\eta}]\right) + \bigO{\frac{1}{r_0^{2\eta}}{\log r_0}}
\end{equation}
\begin{equation}
    = \frac{1+\eta\log(r_0)}{r_0^{\eta}}\mathbb{E}[R^{\eta}] - \frac{1}{r_0^{\eta}}\mathbb{E}[R^{\eta}]\log(\mathbb{E}[R^{\eta}]) + \bigO{\frac{1}{r_0^{2\eta}}\log r_0}
\end{equation}
From Theorem 2 we have
\begin{equation}
    \int_0^Df(r)h_2(p(r)) = \frac{1+\eta\log(r_0)}{r_0^{\eta}}\mathbb{E}[R^{\eta}] - \frac{1}{r_0^{\eta}}\mathbb{E}[R^{\eta}\log(R^{\eta})] + \bigO{\frac{1}{r_0^{2\eta}}\log r_0}
\end{equation}
This means that
\begin{gather}
    \bar{p}\Delta C_n \sim \frac{1+\eta\log(r_0)}{r_0^{\eta}}\mathbb{E}[R^{\eta}] - \frac{1}{r_0^{\eta}}\mathbb{E}[R^{\eta}]\log(\mathbb{E}[R^{\eta}]) \nonumber \\ 
    - \frac{1+\eta\log(r_0)}{r_0^{\eta}}\mathbb{E}[R^{\eta}] + \frac{1}{r_0^{\eta}}\mathbb{E}[R^{\eta}\log(R^{\eta})]
\end{gather}
\begin{equation}
    = \frac{1}{r_0^{\eta}}\left(\mathbb{E}[R^{\eta}\log(R^{\eta})] - \mathbb{E}[R^{\eta}]\log(\mathbb{E}[R^{\eta}]) \right)
\end{equation}
So finally, because $1 - \bar{p} \sim \frac{\mathbb{E}[R^{\eta}]}{r_0^{\eta}}$,
\begin{equation}
    \Delta C_n \sim \frac{\frac{1}{r_0^{\eta}}\left(\mathbb{E}[R^{\eta}\log(R^{\eta})] - \mathbb{E}[R^{\eta}]\log(\mathbb{E}[R^{\eta}]) \right)}{1-\frac{\mathbb{E}[R^{\eta}]}{r_0^{\eta}}}
\end{equation}
\begin{equation}
    = \frac{\Theta(1)}{r_0^{\eta}-\Theta(1)} = \Theta(r_0^{-\eta})
\end{equation}
That is, the compressibility difference goes to 0.
\end{proof}
\label{compressibility_proof}

\section{Section \ref{sec:entropy_graph} Derivations and Proofs}
\label{sec:wedge_and_cantor}
\subsection{Derivation of Entropy Mass Expansion in the Wedge}
\label{sec:wedge}
We will follow the methodology of \cite{dettmann2016random} for our modified connection function $\rho(r) = h_2(p(r))$.
\subsubsection{Integration on a Non-Centred Line}
As briefly described, we must compute the entropy mass $H_x(\graph)$ near each boundary component. For the wedge, this corresponds to the bulk, the edges and the corners. All of these will rely on the following integral, where the entropy mass is integrated along a line passing a small distance $x$ from the node in question.
\begin{equation}
    F(x) = \int_0^{\infty} \rho(\sqrt{x^2+t^2})dt
\end{equation}
For $\epsilon > 0$, we can write this as $F(x) = f(x,\epsilon) + F(x,\epsilon)$ where
\begin{align}
    f(x,\epsilon) = \int_0^{\epsilon} \rho(\sqrt{x^2+t^2}) dt  \nonumber \\
    F(x,\epsilon) = \int_{\epsilon}^{\infty} \rho(\sqrt{x^2+t^2})dt
\end{align}
and we assume that $\epsilon \gg x^k$ for every $k\in \mathbb{N}$. First, we will deal with $f(x,\epsilon)$. Let $s = \sqrt{t^2/x^2 + 1}$, then,
\begin{equation}
    f(x,\epsilon) = \int_1^{\sqrt{\epsilon^2/x^2+1}} \rho(xs)\frac{xs\space ds}{\sqrt{s^2-1}}
\end{equation}
Substituting our expression for $\rho$ with small $r$, we have
\begin{equation}
    f(x,\epsilon) = \epsilon \rho(0) + \sum_{\alpha \in \mathcal{A}}a_{\alpha}\int_1^{\sqrt{\epsilon^2/x^2+1}} (xs)^{\alpha}\frac{xs \space ds}{\sqrt{s^2-1}}+ \sum_{\beta \in \mathcal{B}}b_{\beta}\int_1^{\sqrt{\epsilon^2/x^2+1}} (xs)^{\beta} \log(xs) \frac{xs \space ds}{\sqrt{s^2-1}}
\end{equation}
For the first sum, we can directly use the result of \cite{dettmann2016random}, which gives
\begin{equation}
    \sum_{\alpha \in \mathcal{A}}a_{\alpha}\int_1^{\sqrt{\epsilon^2/x^2+1}} (xs)^{\alpha}\frac{xs \space ds}{\sqrt{s^2-1}} = \sum_{\alpha \in \mathcal{A}} a_{\alpha}f_{\alpha}(x,\epsilon)
\end{equation}
where 
\begin{equation}
    f_{\alpha}(x,\epsilon) = \frac{\sqrt{\pi}}{2}\frac{\Gamma(-\frac{\alpha+1}{2})}{\Gamma(-\frac{\alpha}{2})}x^{\alpha+1} + \frac{\epsilon^{\alpha+1}}{\alpha+1} + \frac{\alpha \epsilon^{\alpha-1}}{2(\alpha-1)}x^2 + \frac{\alpha(\alpha-2)\epsilon^{\alpha-3}}{8(\alpha-3)}x^4 + o(x^4)
\end{equation}
except when $\alpha =1$ or $\alpha =3$, where
\begin{align}
    f_1(x,\epsilon) = \frac{\epsilon^2}{2} + \frac{1}{2}\log\left(\frac{2\sqrt{e}\epsilon}{x}\right)x^2 + \frac{x^4}{16\epsilon^2} + o(x^4) \nonumber \\
    f_3(x,\epsilon) = \frac{\epsilon^4}{4} + \frac{3\epsilon^2}{4}x^2 + \frac{3}{8}\log\left(\frac{2e^{3/4}\epsilon}{x}\right)x^4 + o(x^4)
\end{align}
So we are left to expand the second sum in this manner to get a full expression for $f(x,\epsilon)$. Expanding the denominator in the integral gives
\begin{gather}
    \sum_{\beta \in \mathcal{B}}b_{\beta}\int_1^{\sqrt{\epsilon^2/x^2+1}} (xs)^{\beta} \log(xs) \frac{xs \space ds}{\sqrt{s^2-1}} \nonumber \\ = \sum_{\beta \in \mathcal{B}}b_{\beta}x^{\beta+1}\int_1^{\sqrt{\epsilon^2/x^2+1}} s^{\beta} \log(xs) \left(1+\frac{1}{2s^2} + \frac{3}{8s^4} + \frac{5}{16s^6} + \sum_{k=4}^\infty \frac{(2k)!}{(k!)^22^{2k}}s^{-2k}\right)ds
\end{gather}
which is a sum over integrals of the form $\int s^{\beta-2k}\log(xs)ds$. We therefore must be careful when $\beta-2k = -1$. For now, assume that each $\beta$ is not an odd integer. Then
\begin{align}
    \label{eq:integral_of_Taylor}
    \sum_{\beta \in \mathcal{B}}b_{\beta}\int_1^{\sqrt{\epsilon^2/x^2+1}} (xs)^{\beta} \log(xs) \frac{xs \space ds}{\sqrt{s^2-1}} \nonumber \\ = \sum_{\beta \in \mathcal{B}}b_{\beta}x^{\beta+1}\left[\frac{s^{\beta+1}((\beta+1)\log(xs) - 1)}{(\beta+1)^2} + \frac{s^{\beta-1}((\beta-1)\log(xs) - 1)}{2(\beta-1)^2} \right. \nonumber \\ \left. +\frac{3s^{\beta-3}((\beta-3)\log(xs) - 1)}{8(\beta-3)^2} + o(s^{\beta-3})\right]_{1}^{\sqrt{\epsilon^2/x^2+1}}
\end{align}
Then one can follow the expansion, and collect terms to result in
\begin{align}
     \sum_{\beta \in \mathcal{B}}b_{\beta}\int_1^{\sqrt{\epsilon^2/x^2+1}} (xs)^{\beta} \log(xs) \frac{xs \space ds}{\sqrt{s^2-1}}ds = \sum_{\beta \in \mathcal{B}}b_{\beta}g_{\beta}(x,\epsilon)
\end{align}
where for $\beta \neq 1, 3$
\begin{align}
    g_{\beta}(x,\epsilon) = C_1(\beta)x^{\beta+1} + C_2(\beta)x^{\beta+1}\log(x) - \frac{1}{(\beta+1)^2}\epsilon^{\beta+1} + \frac{1}{\beta+1}\epsilon^{\beta+1}\log \epsilon \nonumber \\
    - \frac{1}{2(\beta-1)^2}x^2\epsilon^{\beta-1} + \frac{\beta}{2(\beta-1)}x^2\epsilon^{\beta-1}\log \epsilon + \frac{\beta^2-6\beta+6}{8(\beta-3)^2}x^4\epsilon^{\beta-3} + \frac{\beta(\beta-2)}{8(\beta-3)}x^4\epsilon^{\beta-3}\log \epsilon + o(x^4) 
\end{align}
with 
\begin{gather}
    C_1(\beta) = \frac{\sqrt{\pi}}{4}\frac{\Gamma(-\frac{\beta+1}{2})}{\Gamma(-\frac{\beta}{2})}\left[\psi^{(0)}\left(-\frac{\beta}{2}\right) - \psi^{(0)}\left(-\frac{\beta+1}{2}\right)\right] \nonumber \\
    C_2(\beta) = \frac{\sqrt{\pi}}{2}\frac{\Gamma(-\frac{\beta+1}{2})}{\Gamma(-\frac{\beta}{2})}
\end{gather}
and where $\psi^{(0)}$ is the digamma function. For $\beta=1,3$ we need to compute the values of $g_\beta$ explicitly, since the general formula in (\ref{eq:integral_of_Taylor}) does not hold (since we integrate $\log(s)/s$). We have
\begin{align}
    g_1(x,\epsilon) &= -\frac{\epsilon^2}{4} + \frac{\epsilon^2 \log \epsilon}{2} + \frac{x^2}{2}\left(\frac{1}{2} + \log \epsilon + \frac{1}{2}\log^2\epsilon + \frac{3}{16}{_4}F_3\left(1,1,1,\frac{5}{2};2,2,3;1\right) \right) \nonumber \\  
    &+\frac{x^2}{2}\log (x) \log\left(\frac{2}{x^{1/2}e^{1/2}}\right) + \frac{1}{32}\frac{x^4}{\epsilon^2}  + \frac{1}{16}\frac{x^4}{\epsilon^2}\log \epsilon - \frac{1}{16}\frac{x^4}{\epsilon^2}\log x + o(x^4) \nonumber \\
    g_3(x,\epsilon) &= -\frac{1}{16}\epsilon^4 + \frac{\epsilon^4 \log \epsilon}{4} - \frac{1}{8}x^2\epsilon^2 + \frac{3}{4}x^2\epsilon^2\log \epsilon  \nonumber \\
    &+ x^4 \left(\frac{1}{2}\log \epsilon + \frac{3}{16}\log^2 \epsilon + \frac{5}{16}+\frac{5}{64}{_4}F_3\left(1,1,1,\frac{7}{2};2,2,4;1\right) \right)  \nonumber \\& + x^4\log x \left(\frac{3}{8}\log 2 - \frac{7}{2}\right)  - \frac{3}{16}x^4 \log^2(x)+ o(x^4)
\end{align}
where ${_p}F_q$ is the hypergeometric function. So
\begin{equation}
    f(x,\epsilon) = \epsilon\rho(0) + \sum_{\alpha \in \mathcal{A}} a_{\alpha}f_{\alpha}(x,\epsilon) + \sum_{\beta \in \mathcal{B}}b_{\beta}g_{\beta}(x,\epsilon)
\end{equation}
For $F(x,\epsilon)$, we may use exactly the same method as \cite{dettmann2016random}, performing a Taylor expansion of $\rho(\sqrt{x^2+t^2})$ under the integral sign by explicitly taking derivatives with respect to $x^2$. For $k \in \mathbb{Z}$, define
\begin{align}
    \rho_k = \int_0^{\infty} r^k\rho(r)dr \nonumber \\
    \rho'_k = \int_0^{\infty} r^k \rho'(r)dr \nonumber \\
    \rho_k'' = \int_0^{\infty} r^k\rho''(r)dr
\end{align}
where $\rho'$ indicates the derivative with respect to $r$, and define their incomplete moments
\begin{equation}
    \rho_k(\epsilon) = \int_{\epsilon}^{\infty} r^k\rho(r)dr
\end{equation}
where the respective moments of the derivatives are defined as expected.
Also, define
\begin{align}
    \Delta_m = \sum_{k \in \mathcal{K}} r_k^m(\rho(r_k+) - \rho(r_k-)) \nonumber \\
    \Delta_m' = \sum_{k \in \mathcal{K}} r_k^m(\rho'(r_k+) - \rho'(r_k-))
\end{align}
where $\rho(r_k+) = \lim_{r\to r_k^+} \rho(r)$ (and similarly for $\rho(r_k-)$) is the contribution from discontinuities. Then we by splitting the integral at the discontinuities and directly integrating we have
\begin{equation}
    F(x,\epsilon) = \rho_0(\epsilon) + \frac{x^2}{2}[\rho'_{-1}(\epsilon) + \Delta_{-1}] + \frac{x^4}{8}[\rho''_{-2}(\epsilon) - \rho'_{-3}(\epsilon) + \Delta'_{-2} + \Delta_{-3}] + o(x^4)
\end{equation}
As in \cite{dettmann2016random}, note that $F(x) = f(x,\epsilon) + F(x,\epsilon)$ has the same value for every $\epsilon$. In particular, we may take the limit $\epsilon\ttz$, and obtain

\begin{alignat}{2}
    F(x) &= \rho_0 +\sum_{\alpha \neq 1,3} a_{\alpha}C_2(\alpha)x^{\alpha+1} + \sum_{\beta \neq 1,3} b_{\beta}(C_1(\beta) + C_2(\beta)\log x)x^{\beta+1} 
    \nonumber \\ &+ \frac{x^2}{2}\left[\tilde{\rho}_1 - a_1\log(|a_1|x) + b_1\log(x) \log\left(\frac{2}{x^{1/2}e^{1/2}}\right) \right] \nonumber \\ &+ \frac{x^4}{8}\Bigg[\tilde{\rho}_3 + b_3\Bigg(\left(3\log(2) - \frac{7}{4} \right)\log x + \frac{5}{2}\nonumber \\&+ \frac{5}{8}{_4}F_3\left(1,1,1,\frac{7}{2};2,2,4;1\right)-\frac{3}{2}\log^2x\Bigg) - 3a_3\log(|a_3|x)\Bigg] + o(x^4\log x)
\end{alignat}

Where
\begin{align}
    \tilde{p}_1 = \lim_{\epsilon \to 0} \left(\rho'_{-1}(\epsilon) + \Delta_{-1} + \sum_{\alpha \neq 1} a_{\alpha}\frac{\alpha\epsilon^{\alpha-1}}{\alpha-1} + \sum_{\beta \neq 1} b_{\beta}\epsilon^{\beta-1}\left(\frac{\beta \log \epsilon}{\beta-1}-\frac{1}{(\beta-1)^2}\right)\right.  \nonumber \\
     \left.+ a_1\log \left({2|a_1|\sqrt{e}\epsilon}\right) + b_1\left(\frac{1}{2} + \log \epsilon + \frac{1}{2}\log^2\epsilon + \frac{3}{16}{_4}F_3\left(1,1,1,\frac{5}{2};2,2,3;1\right)\right)\right)
\end{align}
\begin{align}
    \tilde{p}_3 = \lim_{\epsilon\to 0} \left(\rho''_{-2}(\epsilon) - \rho'_{-3}(\epsilon) + \Delta'_{-2}+\Delta_{-3} + \sum_{\alpha \neq 3} a_{\alpha}\epsilon^{\alpha-3}\frac{\alpha(\alpha-2)}{(\alpha-3)} +3a_3\log\left(2|a_3|e^{3/4}\epsilon\right) \right. \nonumber \\
    \left.+ \sum_{\beta \neq 3} b_{\beta}\epsilon^{\beta-3}\left[\frac{\beta(\beta-2)}{(\beta-3)}\log \epsilon + \frac{\beta^2-6\beta+6}{(\beta-3)^2}\right] +b_3\left[ 4\log\epsilon + \frac{3}{2}\log^2\epsilon \right] \right)
\end{align}
Note these limits must exist because the limit as $\epsilon\to 0$ of $f(x,\epsilon) + F(x,\epsilon)$ exists and is equal to $F(x)$. This gives us the form of $F(x)$ in full generality. Next, we can apply this to wedge.

\subsubsection{Entropy Mass in a Corner}
\begin{figure}
    \centering
    \includegraphics[width=0.8\linewidth]{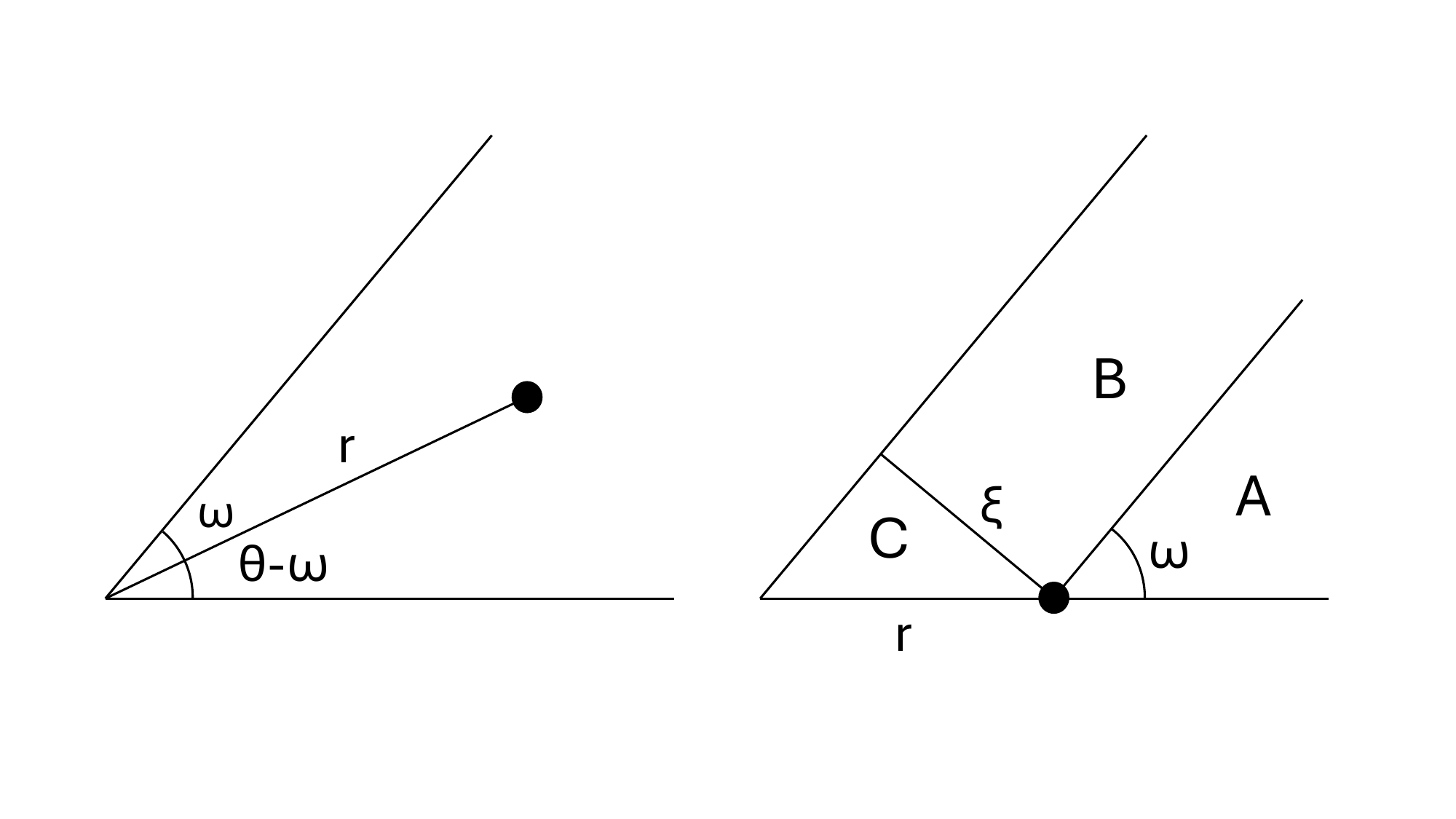}
    \vspace{-4em}
    \caption{Reproduced Figure 2 from \cite{dettmann2016random}, explaining how we split up the wedge to compute the entropy mass. The left panel shows the whole wedge and a polar point $(r, \omega) \in W_{\theta}$. The right panel shows the top part of the wedge so that the point is on the boundary of the wedge $W_{\omega}$. The letters A,B and C indicate the regions that we integrate the entropy mass over in equation \eqref{eq:split_it_up}.}
    \label{fig:wedge_split}
\end{figure}
\label{sec:corner_mass}
Following the method from \cite{dettmann2016random}, but with our form of $F$, we can write the entropy mass of a point $x$ in a wedge with angle $\theta$, denoted $M_{\theta}(x)$ as the sum from three contributions, the bulk, the edges and the corners. First, we will start in a wedge of angle $\omega$ and assume $x$ is on the boundary, and then add on the entropy mass from a wedge of angle $\omega' = \theta - \omega$ to recover the entropy mass of a full point. 
\begin{equation}
    \label{eq:split_it_up}
    H_{x}(\graph) = H_x^A(\graph)+H_x^B(\graph)+H_x^C(\graph)
\end{equation}
Consider the wedge $W_{\omega}$, and put the point $x$ on the lower boundary at a distance $r$ from the corner. Then, with $\xi$ as the perpendicular distance from the point $x$ to the top boundary (see Figure \ref{fig:wedge_split})
\begin{equation}
    H_{\xi}^A(\graph) := \int_0^{\omega}\int_0^\infty \rho(r)rdrd\phi = \omega \rho_1
\end{equation}
\begin{equation}
    H_{\xi}^B(\graph) = \int_0^\xi\int_0^{\infty} \rho(\sqrt{s^2+t^2}) dxdt= \int_0^{\xi}F(x)dx
\end{equation}
\begin{equation}
    H_{\xi}^C(\graph)=\int_0^{\xi}\int_0^{x\cot \omega} \rho(\sqrt{x^2+t^2}) dxdy = \int_0^\xi f(x,x\cot\omega)dx
\end{equation}
These are computable using the expressions for $f$ and $F$ we developed in the previous sections. Indeed
\begin{align}
    H_{\xi}^B(\graph) = \sum_{\alpha\in \mathcal{A}\setminus\{1,3\}} \frac{a_{\alpha}C_2(\alpha)}{\alpha+2}{\xi}^{\alpha+2} + \sum_{\beta \in \mathcal{B}\setminus\{1,3\}} b_{\beta}\xi^{\beta+2}\left(\frac{(\beta+2)C_1(\beta)+(\beta+2)C_2(\beta)\log(\xi) - C_2(\beta)}{(\beta+2)^2}\right)  \nonumber \\
    + \rho_0\xi + \frac{\xi^3}{6}\left[\tilde{\rho}_{-1} -a_1\log(|a_1|\xi)  + b_1\left(\log\left(\frac{2}{\sqrt{\xi}}\right) + \frac{1}{6}\right)\log(\xi) - \frac{1}{18}\left(7-\log 2\right) \right] \nonumber \\
      + \frac{\xi^5}{40}\Big[\tilde{\rho}_{-3}+b_3\left(\left(3\log 2 - \frac{23}{20}\right)\log \xi - \frac{3}{2}\log^2(\xi) - \frac{3}{5}\log 2 + \frac{273}{200} + \frac{5}{8}{_4}F_3\left(1,1,1,\frac{7}{2};2,2,4;1\right)\right) \nonumber \\ - \frac{3a_3}{25} (5\log(|a_3|\xi)-1)\Big] + o(\xi^5)
\end{align}
In $H_{\xi}^C(\graph)$, we do not have $x\cot \omega \gg x$, and therefore the separation between $x$ and $\epsilon$ does not work as before, so we need to integrate the expression for $F(x)$ directly, giving
\begin{equation}
    H_{\xi}^C(\graph) = \int_0^{\xi} (x\cot \omega)\rho(0) + \sum_{\alpha \in \mathcal{A}} a_{\alpha}f_{\alpha}(x,x\cot \omega) + \sum_{\beta \in \mathcal{B}}b_{\beta}g_{\beta}(x,x\cot \omega) dx
\end{equation}
\begin{align}
    = \frac{\xi^2}{2}\rho(0)\cot \omega + \sum_{\alpha \notin \{1,3\}} a_{\alpha}\xi^{\alpha+2}\left[\frac{C_2(\alpha)}{\alpha+2} + \frac{\csc^{\alpha+1}(\omega)}{(\alpha+1)(\alpha+2)}{_2}F_1\left(\frac{1}{2}, \frac{-1-\alpha}{2}; \frac{1-\alpha}{2}; \sin^2 \omega\right)\right] \nonumber \\
    + \sum_{\beta \notin \{1,3\}}b_\beta\xi^{\beta+2} \Bigg[\frac{(\beta+2)\log(\xi) - 1}{(\beta+2)^2}\Bigg(\frac{\csc^{\beta+1}(\omega)}{(\beta+1)}{_2}F_1\left(\frac{1}{2}, -\frac{\beta}{2}-\frac{1}{2}; \frac{1}{2}-\frac{\beta}{2}; \sin^2 \omega\right) + C_2(\beta)\Bigg)
    \nonumber \\
    + \frac{1}{\beta+2}\Bigg( \frac{\csc^{\beta +1}(\omega)}{(\beta+1)^2}{_3}F_2\left(-\frac{1}{2}, -\frac{\beta}{2}-\frac{1}{2}, -\frac{\beta}{2}-\frac{1}{2}; \frac{1}{2}-\frac{\beta}{2}, \frac{1}{2}-\frac{\beta}{2}; \sin^2 \omega\right) \nonumber \\
    - \frac{\csc^{\beta+1}\log(\sin(\omega))}{\beta+1}{_2}F_1\left(-\frac{1}{2}, -\frac{\beta}{2}-\frac{1}{2}; \frac{1}{2}-\frac{\beta}{2}; \sin^2 \omega\right) \nonumber \\
    + \frac{\csc^{\beta-1}(\omega)}{(\beta-1)^2}{_3}F_2\left(-\frac{1}{2}, \frac{1}{2}-\frac{\beta}{2}, \frac{1}{2}-\frac{\beta}{2}; \frac{3}{2}-\frac{\beta}{2}, \frac{3}{2}-\frac{\beta}{2}; \sin^2 \omega\right)\Bigg)
    \nonumber \\ 
    + \frac{\csc^{\beta-1}(\omega)\log(\sin(\omega))}{\beta+1}{_2}F_1\left(\frac{1}{2}, \frac{1}{2}-\frac{\beta}{2}; \frac{3}{2}-\frac{\beta}{2}; \sin^2 \omega \right) \nonumber \\
    -\frac{1}{(\beta-1)^2}{_3}F_2\left(-\frac{1}{2}, \frac{1}{2}-\frac{\beta}{2}, \frac{1}{2}-\frac{\beta}{2}; \frac{3}{2}-\frac{\beta}{2}, \frac{3}{2}-\frac{\beta}{2}; 1\right) \nonumber \\
    - \frac{1}{(\beta+1)^2}{_3}F_2\left(-\frac{1}{2}, -\frac{\beta}{2}-\frac{1}{2}, -\frac{\beta}{2}-\frac{1}{2}; \frac{1}{2}-\frac{\beta}{2}, \frac{1}{2}-\frac{\beta}{2}; 1\right) \Bigg] \nonumber \\ 
    %+ b_1\frac{\xi^3}{3}\left(\frac{1}{2}{_4}F_3\left(1,1,1,\frac{5}{2}; 2, 2, 3 ;1\right) - \frac{1}{3}\log 2 - \frac{17}{36} + \left(\frac{13}{6} + \log(2 \cot \omega) + \cot^2 \omega \right)\log \xi - \log^2 \xi + \right.\nonumber \\ \left. \frac{1}{96}\tan^2 \omega + \frac{1}{16}\tan \omega - \frac{1}{3}\log(\cot \omega) + \frac{1}{4}\log^2 (\cot \omega)   + \frac{1}{2}\cot^2(\omega)\left(\log(\cot \omega) - \frac{7}{6}\right) \right)  \nonumber \\  + b_3\frac{\xi^5}{5}\left(\left(\frac{23}{160}-\frac{3\log 2}{40}\right)\left(5\log \xi - 1\right) + \log(\cot \omega) - \frac{1}{16}\cot^4\omega - \frac{1}{32}\cot^2 \omega \right) \nonumber \\ 
    + b_1\frac{\xi^3}{6}\left[\frac{\pi^2}{24} +  \frac{1}{2}\text{Li}_2\left(-\frac{\sin^2\omega}{(1+\cos\omega)^2}\right) - \frac{1}{2}\log\left(\frac{1+\cos\omega}{\sin\omega}\right)\left(\log\left(\frac{4\sin \omega}{1+\cos \omega}\right)-1\right) \nonumber \right. \\
    \left. + \frac{1}{2}\cot\omega \csc \omega(2\log \csc \omega-1) + \frac{\cos \omega}{\sin \omega}\log \xi + \log\left(\frac{1+\cos\omega}{\sin \omega}\right)\log \xi\right] \nonumber \\
    + b_3 \frac{\xi^5}{40}\left[\frac{\pi^2}{8}+\frac{3}{2}\text{Li}_2\left(-\frac{\sin^2(\omega)}{(1+\cos\omega)^2}\right) - \frac{1}{4}\log\left(\frac{1+\cos \omega}{\sin \omega}\right)\left(6\log\left(\frac{4\sin\omega}{1+\cos\omega}\right) - 7\right) \nonumber \right. \\ \left. + \cot\omega \csc\omega (20\log \csc \omega - 2 \cot^2 \omega (1-4\log \csc \omega)-7) \right. \nonumber \\ \left. + \frac{\cos\omega}{\sin^2\omega}\left(3+\frac{2}{\sin^2\omega}\right)\log \xi + 3\log \left(\frac{1+\cos\omega}{\sin \omega}\right)\log \xi\right] + o(\xi^5\log \xi)
\end{align}
where
\begin{equation}
    \text{Li}_s(z) = \sum_{k=1}^\infty \frac{z^k}{k^s}
\end{equation}
is the polylogarithm function.
% \begin{align}
%     g_1: K\xi^3\left(\frac{\pi^2}{48} +  \frac{1}{4}\text{Li}_2\left(-\frac{\sin^2\omega}{(1+\cos\omega)^2}\right) - \frac{1}{4}\log\left(\frac{1+\cos\omega}{\sin\omega}\right)\left(\log\left(\frac{4\sin \omega}{1+\cos \omega}\right)-1\right) \nonumber \right. \\
%     \left. + \frac{1}{4}\cot\omega \csc \omega(2\log \csc \omega-1)
%     \right)
% \end{align}
% \begin{align}
%     g_3: K_2\xi^5\left(\frac{\pi^2}{64}+\frac{3}{16}\text{Li}_2\left(-\frac{\sin^2(\omega)}{(1+\cos\omega)^2}\right) - \frac{1}{32}\log\left(\frac{1+\cos \omega}{\sin \omega}\right)\left(6\log\left(\frac{4\sin\omega}{1+\cos\omega}\right) - 7\right) \nonumber \right. \\ \left. \frac{1}{8}\cot\omega \csc\omega (20\log \csc \omega - 2 \cot^2 \omega (1-4\log \csc \omega)-7)\right)
% \end{align}
% in $g_3$ take out an eighth and put it directly into the expression.

Then we get, as a function of a polar point $(r,0)$ where $r = \xi \csc \omega$,
\begin{equation}
    H_{(\xi \csc \omega, 0)}(\graph) =  \omega \rho_1 + \rho_0 \xi + \frac{\xi^2}{2}\rho(0)\cot \omega + \bigO{\xi^3, \xi^{\alpha_{\min}+2}, \xi^{\beta_{\min}+2}\log \xi} 
\end{equation}
and so for a general point $(r, \omega)$ in the wedge,
\begin{equation}
    H_{(r, \omega)}(\graph|\Omega=W_\theta) = H_{(r, 0)}(\graph | \Omega=W_{\omega}) + H_{(r, 0)}(\graph|\Omega = W_{\omega'})
\end{equation}
\begin{equation}
    = \theta \rho_1 + \rho_0 r(\sin \omega + \sin \omega') + \rho(0) \frac{r^2}{2}(\sin \omega \cos \omega + \sin \omega' \cos \omega') + \bigO{r^3, r^{\alpha_{\min}+2}, r^{\beta_{\min}+2}} 
\end{equation}
where $\omega' = \theta -\omega$.

\subsection{Cantor Set}
\label{app:cantor}
This section contains the proofs of the Lemmas used in Section \ref{cantor}. 

\subsubsection{Proof of Lemma \ref{thm:recursive}}
\label{lemma1}
\begin{proof}
Any $x \in \mathcal{C}$ may be written as an $\lambda$-ary expansion
\begin{equation}
    x = \sum_{n=1}^{\infty}\frac{x_n}{\lambda^n}
\end{equation}
where $x_n\in\{0, \lambda-1\}$. Therefore, any distance between two points may be written as
\begin{equation}
    r = x-y = \sum_{n=1}^{\infty} \frac{x_n-y_n}{\lambda^n} = \sum_{n=1}^{\infty}\frac{w_n}{\lambda^n}
\end{equation}
where $w_n \in \{1-\lambda, 0, \lambda -1\}$. Now we can write the CDF of the distance $R = X-Y$ with $X,Y$ uniformly distributed in the Cantor set as
\begin{equation}
    \prob{R < r} = \prob{\sum_{n=1}^{\infty} \frac{W_n}{\lambda^n} < r}
\end{equation}
where now $W_n = 1-\lambda, 0, \lambda-1$ independently with probability $1/4, 1/2,1/4$ respectively. Now since this sum is absolutely convergent, we can express this probability as
\begin{equation}
    \prob{R<r} = \prob{\frac{W_1}{\lambda} + \sum_{n=2}^{\infty}\frac{W_n}{\lambda^n} < r}
\end{equation}
\begin{equation}
    = \prob{\sum_{n=1}^{\infty} \frac{W_n}{\lambda^n} < \lambda r - W_1} = \prob{R<\lambda r-W_1}
\end{equation}
\begin{equation}
    = \frac{1}{2}\prob{R < \lambda r} + \frac{1}{4}\prob{R < \lambda r + (\lambda -1)} + \frac{1}{4}\prob{R < \lambda r - (\lambda - 1)}
\end{equation}
which gives us the self similarity relation we wanted.
\end{proof}

\subsubsection{Proof of Lemma \ref{thm:moments}}
\label{lemma3}
The proof of this result uses the change of variables formula for Lebesgue integration.
\begin{lemma}[Change of Variables Formula \cite{bogachev2007measure}]
    A measurable function $g$ on $X_2$ is integrable with respect to the pushforward measure $f*\mu = \mu \circ f^{-1}$ if and only if the composition $g \circ f$ is integrable with respect to the measure $\mu$. In that case, the integrals coincide, i.e.,
    \begin{equation}
        \int_{X_2} g \space d(f*\mu) = \int_{X_1} g\circ f \space d\mu
    \end{equation}
    where $X_1 = f^{-1}(X_2)$
\end{lemma}
\begin{proof}[Proof of Lemma \ref{thm:moments}]
    We have
    \begin{equation}
        C[F;-s] = \int_0^1 r^{-s}dF(r) = \frac{1}{2}\int_0^1 r^{-s}dF(\lambda r) + \frac{1}{4}\int_0^1 r^{-s}dF(\lambda r + (\lambda-1)) + \frac{1}{4}\int_0^1 r^{-s}dF(\lambda r - (\lambda -1 ))
    \end{equation}
    Using the change of variables formula, we can re-write each integral in turn:
    \begin{equation}
        \int_0^1 r^{-s}dF(\lambda r) = \int_0^{\lambda} (\lambda r)^{-s} dF(r) = \frac{1}{\lambda^{-s}}\int_0^1r^{-s}dF(r)
    \end{equation}
    \begin{equation}
        = \frac{1}{\lambda^{-s}}C[F;-s]
    \end{equation}
    where the second equality follows because $dF(r) = 0$ when $r > 1$ (and $\lambda>2$ by its definition). Then, the second integral is
    \begin{equation}
        \int_0^1 r^{-s}dF(\lambda r + (\lambda-1)) = \int_{\lambda-1}^{2\lambda-1} \left(\frac{r - (\lambda -1)}{\lambda}\right)^s dF(r) = 0
    \end{equation}
    since $\lambda-1 > 1$. Finally,
    \begin{equation}
        \label{eq:the_saviour}
        \int_0^1 r^{-s}dF(\lambda - (\lambda -1)) = \int_{1-\lambda}^1 \left(\frac{r + (\lambda - 1)}{\lambda}\right)^{-s}dF(r) = \frac{1}{\lambda^{-s}}\int_{-1}^1(r+(\lambda-1))^{-s}dF(r)
    \end{equation}
    We can further write this as
    \begin{equation}
        \frac{1}{\lambda^{-s}}\int_{-1}^1(r+(\lambda-1))^{-s}dF(r) = \frac{1}{\lambda^{-s}}\left[\int_{-1}^0 (r+(\lambda-1))^{-s}dF(r) + \int_0^1 (r+(\lambda-1))^{-s}dF(r)\right]
    \end{equation}
    \begin{equation}
        \label{eq:integral_difference}
        = \frac{1}{\lambda^{-s}}\left[\int_{0}^1 (-r+(\lambda-1))^{-s}dF(r) + \int_0^1 (r+(\lambda-1))^{-s}dF(r)\right]
    \end{equation}
    where we have used the change of variable formula and the fact that $dF(r) = dF(-r)$ again on the first integral. Now, note that $(\lambda-1)-r$ is positive for $r \in[0,1]$, as is $r+(\lambda-1)$. Therefore, we can perform a Taylor expansion of the integrands as follows:
    \begin{equation}
        (r+(\lambda-1))^{-s} = (\lambda-1)^{-s}\left(\frac{r}{\lambda-1} + 1\right)^{-s}
    \end{equation}
    \begin{equation}
        = \sum_{l=0}^{\infty} \binom{-s}{l}r^{l}(\lambda-1)^{-s-l}
    \end{equation}
    and likewise
    \begin{equation}
        (-r+(\lambda-1))^{-s} = \sum_{l=0}^{\infty} \binom{-s}{l}(-r)^{l}(\lambda-1)^{-s-l}
    \end{equation}
    where we have the generalised binomial coefficient
    \begin{equation}
        \binom{-s}{l} = \frac{(-s)(-s-1)\cdots(-s-l+1)}{l!} = \frac{(-s)_l}{l!}
    \end{equation}
    where $(\cdot)_l$ is the Pochhammer symbol. So, substituting back into (\ref{eq:integral_difference}) gives
    \begin{equation}
        \frac{1}{\lambda^{-s}}\int_{-1}^1(r+(\lambda-1))^{-s}dF(r) = \frac{1}{\lambda^{-s}}\sum_{l=0}^{\infty}\binom{-s}{l}(\lambda-1)^{-s-l}\int_0^1\left[r^{l} + (-r)^l\right]dF(r)
    \end{equation}
    \begin{equation}
        =  \frac{1}{\lambda^{-s}}\sum_{l=0}^{\infty}\binom{-s}{l}(\lambda-1)^{-s-l}(C[F;l] + (-1)^{l}C{[F;l])}
    \end{equation}
    Combining the three integrals together gives
    \begin{equation}
        C[F;-s] = \frac{1}{2\lambda^{-s}}C[F;-s] + \frac{1}{4\lambda^{-s}}\sum_{l=0}^{\infty}\binom{-s}{l}(\lambda-1)^{-s-l}(C[F;l] + (-1)^{l}C{[F;l])}
    \end{equation}
    and rearranging gives
    \begin{equation}
        C[F;-s] = \frac{1}{(2\lambda^{-s}-1)}\sum_{l=0}^{\infty}\binom{-s}{2l}(\lambda-1)^{-s-2l}C[F;2l]
    \end{equation}
    It remains to show that this sum is well-defined, that is, convergent. First, it is easy to see that for any positive integer $l$,
    \begin{equation}
        C[F;l] = \int_0^1 r^ldF(r) \leq \int_0^1dF(r) = \frac{1}{2}
    \end{equation}
    and so
    \begin{equation}
        \sum_{l=0}^{\infty}\left|\binom{-s}{2l}(\lambda-1)^{-s-2l}C[F;2l]\right| \leq \sum_{l=0}^{\infty}\left|\binom{-s}{2l}\right|(\lambda-1)^{-\Re(s)-2l}
    \end{equation}
    now for a fixed $s$, we have
    \begin{equation}
        (\lambda-1)^{-\Re(s)-2l}\left|\binom{-s}{2l}\right| \sim \frac{(-1)^{2l}}{\Gamma(s)} (2l)^{\Re(s)-1}(\lambda-1)^{-\Re(s)-2l}
    \end{equation}
    which decays exponentially fast as $l\tti$. Therefore, the sum converges absolutely.
\end{proof}

\subsubsection{Proof of Theorem \ref{thm:cantor_entropy_expression}}
\label{sec:residue}
In this section, we detail the derivation of the expression for the conditional entropy using the residue theorem. As mentioned, we will need a technical assumption on the connection function.

\begin{assumption}
\label{eq:moment_bounded}
Assume $p$ is a connection function satisfying
    \begin{enumerate}
    \item \begin{equation}
        \int_0^{\infty} u^{s-1}h_2(p(u))du < \infty
    \end{equation}
    \item \begin{equation}
        \lim_{\substack{B\tti \\ A \ttz}}\left[u^sh_2(p(u))\right]_{u=A}^{u=B} = 0
    \end{equation}
    \item \begin{equation}
        \int_0^{\infty}u^{s}\left|\frac{d}{du}h_2(p(u))\right|du < \infty
    \end{equation}
    \item \begin{equation}
         \lim_{\substack{B\tti \\ A \ttz}}\left[u^s\frac{d}{du}h_2(p(u))\right]_{u=A}^{u=B} = 0
    \end{equation}
    \item \begin{equation}
        \int_0^{\infty} u^{s+1}\left|\frac{d^2}{du^2}h_2(p(u))\right|du < \infty
    \end{equation}
\end{enumerate}
for every real fixed $s > 0$.
\end{assumption}

Recall that we derived
\begin{equation}
    \overline{H}(\graph(r_0)|\mathcal{R}) = \frac{1}{\pi i}\int_{c_l-i\infty}^{c_l+i\infty} \mathcal{M}[\rho;s]C[F;-s]ds
\end{equation}
where $c_l\in \left[0, \frac{\log 2}{\log \lambda}\right)$ is some real number in the common strip of analyticity of $\mathcal{M}$ and $C$. We will use the residue theorem to find an expression for this integral. Let $T_m = \frac{2\pi}{\log \lambda}\left(m+\frac{1}{2}\right)$, and define the sequence of rectangular contours $\gamma^{(T_m)} = \gamma^{(T_m)}_1 + \gamma^{(T_m)}_2 + \gamma^{(T_m)}_3 + \gamma_4^{(T_m)}$, where $\gamma_1^{(T_m)}$ and $\gamma_3^{(T_m)}$ are the vertical lines connecting $c_l-iT_m$ and $c_l+iT_m$ and $c_r-iT_m$ and $c_r+iT_m$ where $c_r > \frac{\log 2}{\log \lambda}$ respectively, and $\gamma_2^{(T_m)}$ and $\gamma_4^{(T_m)}$ are the upper and lower horizontal lines connecting the ends of $\gamma_1^{(T_m)}$ and $\gamma_2^{(T_m)}$. Then, with $s_m = \frac{\log 2}{\log \lambda} + \frac{2\pi i}{\log \lambda}m$ being the poles of $M[\rho;s]C[F;-s]$, we have
\begin{equation}
    \sum_{m=-\infty}^{\infty} \text{Res}(\mathcal{M}[\rho;s]C[F;-s]; s=s_m) = \frac{1}{2\pi i}\lim_{m\tti}\int_{\gamma^{(T_m)}} \mathcal{M}[\rho;s]C[F;-s]ds
\end{equation}
We aim to show that the integrals over $\gamma_2^{(T_m)}$ and $\gamma_4^{(T_m)}$ tend to 0. 
\begin{lemma}
    \label{thm:decay}
    For any $z = s+iT \in \mathbb{C}$ with $\Re(s) > 0$, we have
    \begin{equation}
        \mathcal{M}[\rho; z] = \bigO{|T|^{-2}}
    \end{equation}
    as $T\tti$.
\end{lemma}
\begin{proof}
    We have
    \begin{equation}
        \mathcal{M}[\rho; s+iT] = r_0^{s+iT}\int_0^{\infty}u^{s-1+iT}h_2(p(u))du
    \end{equation}
    and so, using integration by parts,
    \begin{equation}
        |\mathcal{M}[\rho; s+iT]| = r_0^s\left|\int_0^{\infty}u^{s-1+iT}h_2(p(u))du\right|
    \end{equation}
    \begin{equation}
        = \frac{r_0^s}{s+iT}\left|\lim_{\substack{B\tti \\ A\ttz}}\left[u^{s+iT}h_2(p(u))\right]_{u=A}^{a=B} - \int_0^\infty u^{s+iT}\frac{d}{du}h_2(p(u))du\right|
    \end{equation}
    \begin{align}
        = \frac{r_0^s}{(s+iT)(s+1+iT)}\left|\lim_{\substack{B\tti \\ A\ttz}}\left[u^{s+1+iT}\frac{d}{du}h_2(p(u))\right]_{u=A}^{a=B} - \int_0^\infty u^{s+1+iT}\frac{d^2}{du^2}h_2(p(u))du\right|
    \end{align}
    \begin{equation}
        = \bigO{|T|^{-2}}
    \end{equation}
\end{proof}

\begin{lemma}
    We have
    \begin{equation}
        \lim_{m \tti}\int_{\gamma_2^{(T_m)}}\mathcal{M}[\rho;s]C[F;-s]ds = \lim_{m \tti}\int_{\gamma_4^{(T_m)}}\mathcal{M}[\rho;s]C[F;-s]ds = 0
    \end{equation}
\end{lemma}
\begin{proof}
    First, by Lemma \ref{thm:moments},
    \begin{equation}
        \int_{c_l+iT_m}^{c_r+iT_m} \mathcal{M}[\rho;s]C[F;-s]ds = \int_{c_l}^{c_r} \frac{\mathcal{M}[\rho;s+iT_m]}{2\lambda^{-s-iT_m}-1}\sum_{l=0}^\infty \binom{-s-iT_m}{2l}(\lambda-1)^{-s-iT_m-2l} C[F; 2l]ds
    \end{equation}
    So
    \begin{align}
        &\left|\int_{c_l+iT_m}^{c_r+iT_m} \mathcal{M}[\rho;s]C[F;-s]ds\right| \nonumber \\ &\leq \int_{c_l}^{c_r} |\mathcal{M}[\rho;s+iT_m]|\frac{1}{|2\lambda^{-s-iT_m}-1|}\left|\sum_{l=0}^\infty \binom{-s-iT_m}{2l}(\lambda-1)^{-s-iT_m-2l} C[F; 2l]\right|ds
    \end{align}
    \begin{equation}
        \leq \int_{c_l}^{c_r} \frac{|\mathcal{M}[\rho;s+iT_m]|}{|2\lambda^{-s-iT_m}-1|}\sum_{l=0}^{\infty}\left|\binom{-s}{2l}(\lambda-1)^{-s-2l} C[F;2l]\right|ds
    \end{equation}
    \begin{equation}
        \leq \int_{c_l}^{c_r} \frac{|\mathcal{M}[\rho;s+iT_m]|}{|2\lambda^{-s-iT_m}-1|}\sum_{l=0}^{\infty}\left|\binom{-s}{2l}(\lambda-1)^{-s-2l}\right| ds
    \end{equation}
    We know the sum over $l$ is convergent, so let $S_{\max} := \sup_{s \in [c_l, c_r]} \sum_{l=0}^{\infty} \left|\binom{-s}{2l}(\lambda-1)^{-s-2l}\right| < \infty$.
    Then
    \begin{equation}
        \left|\int_{-\frac{1}{2}+iT_m}^{c+iT_m} \mathcal{M}[\rho;s]C[F;-s]ds\right| \leq S_{\max}\int_{c_l}^{c_r} \frac{|\mathcal{M}[\rho;s+iT_m]|}{|2\lambda^{-s-iT}-1|}ds
    \end{equation}
    For each $m\in\mathbb{Z}$, it is simple to show that $|2\lambda^{-s-iT_m}-1|^{-1}$ is increasing on $s\in[c_l,c_r]$, and therefore is bounded above by $|2\lambda^{-c_r-iT_m}-1|^{-1}$. So, the integrals in question are bounded above by the integral of $|\mathcal{M}[\rho; s \pm iT_m]|$ over the (finite) interval $s \in [c_l, c_r]$. But, we showed that $|\mathcal{M}[\rho; s\pm iT_m]|$ decays at least as fast as $|T_m|^{-2}$ in $m$, so
    \begin{equation}
        \left|\int_{c_l+iT_m}^{c_r+iT_m} \mathcal{M}[\rho;s]C[F;-s]ds\right| \leq \bigO{|T_m|^{-2}}
    \end{equation}
    as $m\tti$. Since we did not assume that $m$ was positive or negative, we have that the limits of the integrals over both $\gamma_2^{T_m}$ and $\gamma_3^{T_m}$ are 0.
\end{proof}

\begin{lemma}
    The sum over residues 
    \begin{equation}
        \sum_{m=-\infty}^{\infty} \text{Res}\left(\mathcal{M}[\rho;s]C[F;-s]; s=s_m\right)
    \end{equation}
    is given by
    \begin{equation}
        \label{eq:residues}
        \sum_{m=-\infty}^{\infty} r_0^{s_m}\psi(s_m)\frac{1}{2\log\lambda}\int_0^1 (r+\lambda-1)^{-s_m}dF(r)
    \end{equation}
    and converges absolutely. Here, $\psi(s) = \int_0^{\infty} u^{s-1}h_2(p(u))dt$. 
\end{lemma}
\begin{proof}
    First, we use a simple change of variables to get
    \begin{equation}
        \mathcal{M}[\rho;s] = r_0^s\psi(s)
    \end{equation}
    where $\psi(s)$ is as in the statement of the Lemma. Then, from the proof of Lemma \ref{thm:moments}, we have
    \begin{equation}
        C[F;-s] = \frac{1}{2(2\lambda^{-s}-1)}\int_0^1\left[(r+\lambda-1)^{-s}+(-r+\lambda-1)^{-s}\right]dF(r)
    \end{equation}
    Then, to compute the residues we must evaluate
    \begin{equation}
        \text{Res}(\mathcal{M}[\rho;s]C[F;-s]; s=s_m) = \lim_{z\to s_m} (z-s_m)\frac{r_0^z\psi(z)}{2(2\lambda^{-z}-1)}\int_0^1\left[(r+\lambda-1)^{-z}+(-r+\lambda-1)^{-z}\right]dF(r)
    \end{equation}
    which evaluates to the summand in (\ref{eq:residues}). Then, to prove absolute convergence let us consider the absolute value of each term in the sum. We have
    \begin{gather}
        \left|\int_0^1(r+\lambda-1)^{-s_m}dF(r)\right| \leq \int_0^1 (r+\lambda-1)^{-\frac{\log 2}{\log \lambda}}dF(r) \leq (\lambda-1)^{-\frac{\log 2}{\log \lambda}} \nonumber \\
        \left|\int_0^1(-r+\lambda-1)^{-s_m}dF(r)\right| \leq \int_0^1(-r+\lambda-1)^{-\frac{\log 2}{\log \lambda }}dF(r) \leq (\lambda -2)^{-\frac{\log 2}{\log \lambda}}
    \end{gather}
    where both bounds are finite since $\lambda$ is strictly greater than 2. So it remains to control
    \begin{equation}
        |r_0^{s_m}\psi(s_m)| = r_0^{\frac{\log 2}{\log \lambda}}|\psi(s_m)|
    \end{equation}
    as $|m|\tti$. But
    \begin{equation}
        |\psi(s_m)| = \left|\int_0^{\infty} u^{\frac{\log 2}{\log \lambda}-1+\frac{2\pi i}{\log \lambda}m} h_2(p(u))du\right|
    \end{equation}
    which we showed is $\bigO{|m|^{-2}}$. Therefore, there exist constants $C, M>0$ such that $|m| > M$ implies $|\psi(s_m)| \leq C|m|^{-2}$. In turn
    \begin{align}
       & \left|\sum_{m=-\infty}^{\infty} r_0^{s_m}\psi(s_m)\frac{1}{2\log\lambda}\int_0^1 \left[(r+\lambda-1)^{-s_m} + (-r+\lambda-1)^{-s_m}\right] dF(r)\right| \nonumber \\ &\leq \sum_{m=-M}^M r_0^{\frac{\log 2}{\log \lambda}}|\psi(s_m)|\frac{1}{2\log\lambda}\int_0^1 \left[(r+\lambda-1)^{-\frac{\log 2}{\log \lambda}} + (-r+\lambda-1)^{-\frac{\log 2}{\log \lambda}}\right] dF(r) \nonumber \\ &+ \left(\frac{r_0}{\lambda-1}\right)^{\frac{\log 2}{\log \lambda}}\sum_{|m| > M} C|m|^{-2}
        < \infty
    \end{align}
\end{proof}

\begin{lemma}
    \label{thm:final_finiite_bound}
    Let $c_r > d$. Define
    \begin{equation}
        I_m(s) = \int_{-T_m}^{T_m}\int_0^\infty r_0^{it}u^{s-1+it}h_2(p(u))dudt
    \end{equation}
    We have
    \begin{equation}
        \lim_{m\tti} I_m(c_r) < \infty
    \end{equation}
\end{lemma}

\begin{proof}
    We will split the integral over $t$ into two pieces. Then use our tail control of $\psi$ to bound the integral. With $\psi(s)=\int_0^\infty u^{s-1}h_2(p(u))$, Lemma \ref{thm:decay} tells us that $\psi(c_r\pm it) = \bigO{|t|^{-2}}$ as $t\tti$. Therefore, for $|t|$ large enough, say $|t| > t^*$, 
    there exists a $C \in \mathbb{R}$, which we may choose to be greater than $\frac{1}{2}$ for later, so that $\psi(c_r\pm it) \leq C|t|^{-2}$.
    Define
    \begin{equation}
        I_m^1(s) = \int_{-t^*}^{t^*}\int_0^\infty r_0^{it}u^{s-1+it}h_2(p(u))dudt
    \end{equation}
    \begin{equation}
        I_m^2(s)=\int_{T_m \geq |t|>t^*}\int_0^\infty r_0^{it}u^{s-1+it}h_2(p(u))dudt
    \end{equation}
    We begin with $I_m^2(c_r)$. We can write
    \begin{equation}
        |I_m^2(c_r)| \leq C\int_{T_m \geq |t|>t^*} \left|\frac{r_0^{it}}{t^2}\right|dt
    \end{equation}
    \begin{equation}
     = 2C\int_{t^*}^{T_m}t^{-2}dt
    \end{equation}
    \begin{equation}
        = 2C\left(\frac{1}{t^*}-\frac{1}{T_m}\right) 
    \end{equation}
    and so $I_m^2(c_r\pm iT)$ is finite, and the limit is less than $2C/t^*$. For $I_m^t(s)$, we have by Fubini's theorem (since the Mellin transform $\psi(c_r\pm it)$ exists for any $t$), 
    \begin{equation}
        I_m^1(c_r) = \int_{-t^*}^{t^*}r_0^{it}\psi(c_r+it)dt
    \end{equation}
    \begin{equation}
        = \int_0^{\infty} u^{c_r-1}h_2(p(u))\int_{-t^*}^{t^*} (ur_0)^{it}dtdu
    \end{equation}
    \begin{equation}
        = \int_0^{\infty}\frac{u^{c_r-1}h_2(p(u))}{\log(r_0u)} \sin(t^*\log (r_0))du
    \end{equation}
    Now since $\frac{\sin(t^*\log(r_0u))}{\log(r_0u)} < t^*$, we have
    \begin{equation}
        I_m^1(c_r) \leq {t^*} \int_{0}^{\infty} u^{c_r-1}h_2(p(u))du < \infty
    \end{equation}
    by assumption. Combining the two results, we have that
    \begin{equation}
        |I_m(c_r)| \leq |I_m^1(c_r)| + |I_m^2(c_r)| < 2C\left(\frac{1}{t^*} - \frac{1}{T_m}\right) + t^*\int_0^{\infty}u^{c_r-1}h_2(p(u))du
    \end{equation}
    \begin{equation}
        \overset{m\tti}{\to}\frac{2C}{t^*}+{t^*}\int_0^{\infty}u^{c_r-1}h_2(p(u))du < \infty
    \end{equation}
\end{proof}
We can finally prove Theorem \ref{thm:cantor_entropy_expression}
\begin{proof}[Proof of Theorem \ref{thm:cantor_entropy_expression}]
    First, from the residue theorem,
    \begin{align}
        \overline{H}(\graph(r_0)|\mathcal{R}) = -\frac{1}{i\pi}\lim_{m\tti}\int_{\gamma_2^{(T_m)}}\mathcal{M}[\rho; s]C[F;-s]ds+\frac{1}{i\pi}\lim_{m\tti}\int_{\gamma_4^{(T_m)}}\mathcal{M}[\rho; s]C[F;-s]ds \nonumber \\
        + \lim_{m\tti} \frac{1}{i\pi}\int_{c_r-iT_m}^{c_r+iT_m} \mathcal{M}[\rho; s]C[F;-s]ds \nonumber \\ + 2\sum_{m=-\infty}^{\infty}r_0^{s_m}\psi(s_m)\frac{1}{2\log\lambda}\int_0^1 \left[(r+\lambda-1)^{-s_m}+(-r+\lambda-1)^{-s_m}\right]dF(r)
    \end{align}
    We know the limit of the integrals over $\gamma_2^{(T_m)}$ ad $\gamma_4^{(T_m)}$ are 0. Thus,
        \begin{align}
    \label{eq:entropy_as_contours}
        \overline{H}(\graph(r_0)|\mathcal{R}) = 
        \lim_{m\tti} \frac{1}{i\pi}\int_{\gamma_3^{(T_m)}} \mathcal{M}[\rho; s]C[F;-s]ds \nonumber \\ + 2\sum_{m=-\infty}^{\infty}r_0^{s_m}\psi(s_m)\frac{1}{2\log\lambda}\int_0^1 \left[(r+\lambda-1)^{-s_m}+(-r+\lambda-1)^{-s_m}\right]dF(r)
    \end{align}
    Since the sum over residues converges absolutely, we may rewrite the sum as
    \begin{align}
        \sum_{m=-\infty}^{\infty}r_0^{s_m}\psi(s_m)\frac{1}{2\log\lambda}\int_0^1 \left[(r+\lambda-1)^{-s_m}+(-r+\lambda-1)^{-s_m}\right]dF(r) \nonumber \\ 
        = r_0^d\psi(d)\frac{1}{2\log \lambda}\int_0^1\left[(r+\lambda-1)^{-d}+(-r+\lambda-1)^{-d}\right]dF(r) \nonumber \\ + r_0^d\sum_{m=1}^{\infty}\left[r_0^{\frac{2\pi i}{\log \lambda}m}\psi(s_m)\frac{1}{2\log\lambda}\int_0^1 \left[(r+\lambda-1)^{-s_m}+(-r+\lambda-1)^{-s_m}\right] dF(r)\right. \nonumber \\ \left. +r_0^{-\frac{2\pi i}{\log \lambda}m}\psi(s_m)\frac{1}{2\log\lambda}\int_0^1 \left[(r+\lambda-1)^{-s_m}+(-r+\lambda-1)^{-s_m}\right]dF(r)\right] \nonumber \\
        = r_0^d \frac{R_0}{2} + r_0^d\sum_{m=1}^{\infty} R_m\left(\frac{e^{i\theta_m + 2\pi i m\frac{\log r_0}{\log \lambda}}+e^{-i\theta_m - 2\pi i m\frac{\log r_0}{\log \lambda}}}{2}\right) \nonumber \\
        = r_0^d \frac{R_0}{2} + r_0^d\sum_{m=1}^{\infty} R_m\cos\left(\theta_m + \frac{2\pi m \log r_0}{\log \lambda}\right)
    \end{align}
    where, with $g(m) = \frac{\psi(s_m)}{\log \lambda}\int_0^1\left[(r+\lambda-1)^{-s_m}+(-r+\lambda-1)^{-s_m}\right]dF(r)$, $R_m = |g(m)|$, and $\theta_m = \arg(g(m))$. Now, 
    \begin{equation}
        \lim_{m\tti} \frac{1}{i\pi}\int_{\gamma_2^{(T_m)}+\gamma_3^{(T_m)}+\gamma_4^{(T_m)}} \mathcal{M}[\rho; s]C[F;-s]ds \leq \lim_{m\tti} \frac{1}{i\pi}\int_{\gamma_3^{(T_m)}} \mathcal{M}[\rho; s]C[F;-s]ds
    \end{equation}
    \begin{equation}
        \leq r_0^{c_r}\left(\frac{2C}{t^*} + t^*\int_0^{\infty}u^{c_r-1}h_2(p(u))du\right)
    \end{equation}
    where we have used Lemma \ref{thm:final_finiite_bound}. Since $c_r >d$, this upper bound is $o(r_0^d)$, and so substituting back in to \eqref{eq:entropy_as_contours} gives
    \begin{equation}
        \overline{H}(\graph(r_0)|\mathcal{R}) = r_0^dR_0 + 2r_0^d \sum_{m=1}^\infty R_m \cos \left(\theta_m + \frac{2\pi m \log r_0}{\log \lambda}\right) + o(r_0^d)
    \end{equation}
    as required.
\end{proof}

\printbibliography

\end{document}

%% file: macro.tex
\usepackage{graphicx} % Required for inserting images
\usepackage{hyperref}
\usepackage[nottoc,numbib]{tocbibind}
\usepackage{amsmath}
\usepackage{amsthm}
\usepackage{amssymb}
\usepackage{amsfonts}
\usepackage{amsmath}
\usepackage{amsthm}
\usepackage{amssymb}
\usepackage{bbm}
\usepackage{tikz}
\usepackage{breqn}
\usepackage{pdflscape}
\usepackage{algorithm}
\usepackage{algpseudocode}
\usepackage{multicol}
\usepackage{tabularx}
\usepackage{blindtext}
\usepackage[]{caption}
\usepackage{hyperref}
\usepackage[nottoc,numbib]{tocbibind}
\usepackage{authblk}
\usepackage[giveninits=true, doi=false, url=false, date=year, backend=biber]{biblatex}
\usepackage[a4paper,
            bindingoffset=0.2in,
            left=.8in,
            right=.8in,
            top=1in,
            bottom=.8in]{geometry}

\newtheorem{definition}{Definition}
\newtheorem{lemma}{Lemma}
\newtheorem{theorem}{Theorem}

\newtheorem{assumption}{Assumption}

\newcommand{\rayleigh}[2]{\exp\left(-\left(\frac{#1}{r_0}\right)^{#2}\right)}
\newcommand{\tti}{\rightarrow\infty}
\newcommand{\graph}{\mathcal{G}}
\newcommand{\bigO}[1]{\mathcal{O}\left(#1\right)}

\newcommand{\prob}[1]{\mathbb{P}\left(#1\right)}

\newcommand{\ttz}[0]{\rightarrow 0}